\documentclass[10pt]{amsart}
\usepackage{amsmath,amsfonts,amsthm,amssymb}
\usepackage{graphicx}
\usepackage{epstopdf}

\theoremstyle{plain}
\newtheorem{theorem}{Theorem}[section]

\newtheorem{lemma}[theorem]{Lemma}
\newtheorem*{de-lemma}{Lemma}

\theoremstyle{remark}

\theoremstyle{definition}

\newcommand{\dd}{\mathrm{d}}
\newcommand{\R}{\mathbb{R}}\newcommand{\N}{\mathbb{N}}

\DeclareMathOperator{\supp}{supp}

\DeclareMathOperator{\nn}{{\bf n}}
\DeclareMathOperator{\ee}{{\bf e}}
\DeclareMathOperator{\VV}{{\bf V}}
\begin{document}

\title[Connecting orbits in Hilbert spaces]{Connecting orbits in Hilbert spaces and applications to P.D.E.}

\author{Panayotis Smyrnelis} \address[P.~ Smyrnelis]{Institute of Mathematics,
Polish Academy of Sciences, 00-656 Warsaw, Poland}
\email[P. ~Smyrnelis]{psmyrnelis@impan.pl}

\date{}

\maketitle
\begin{abstract}
We prove a general theorem on the existence of heteroclinic orbits in Hilbert spaces, and present a method to reduce the solutions of some P.D.E. problems to such orbits. In our first application, we give a new proof in a slightly more general setting of the heteroclinic double layers (initially constructed by Schatzman \cite{scha}), since this result is particularly relevant for phase transition systems. In our second application, we obtain a solution of a fouth order P.D.E. satisfying similar boundary conditions.

\end{abstract}

\section{Introduction and Statements}

Functional Analysis methods are often useful to solve efficiently P.D.E. problems.
We refer to \cite[Ch. 10]{brezis} and \cite[Ch. 7 and 9]{evans} for some classical applications to evolution equations.
The idea is to view a solution $\R^2\ni (t,x)\mapsto u(t,x)$ of a P.D.E. as a map $t\mapsto [U(t):x\mapsto [U(t)](x):=u(t,x)]$ taking its values in a space of functions, and reduce the initial P.D.E. to an O.D.E. problem for $U$.
For instance, in the case of the heat equation and the wave equation, this reduction is based on the theorem of Hille-Yosida \cite[Ch. 10]{brezis} .

In this paper, we apply this viewpoint to the elliptic system
\begin{equation}\label{system0}
\Delta u(t,x)=\nabla W(u(t,x)), \ u:\R^2\to\R^m\  (m\geq 2), \, (t,x)\in \R^2,
\end{equation}
where $W:\R^m\to\R$ is a function such that
\begin{subequations}\label{w13}
\begin{equation}\label{w1}
\text{$W\in C^{2,\alpha}(\R^m; \R)$ (with $\alpha\in (0,1)$) is nonnegative, and has exactly $2$ zeros $a^-$ and $a^+$,}
\end{equation}
\begin{equation}\label{w2}
\text{$D^2W(u)(\nu,\nu)\geq c$, $\forall u\in\R^m$: $|u-a^\pm|\leq r$, $\forall \nu \in \R^m$: $|\nu|=1$, with $r,c>0$},
\end{equation}
\begin{equation}\label{w3}
\liminf_{|u|\to\infty} W(u)>0.
\end{equation}
\end{subequations}
That is, $W$ is a double well potential \eqref{w1}, with nondegenerate minima \eqref{w2}, satisfying moreover the standard asymptotic condition \eqref{w3} to ensure the boundedness of finite energy orbits.
To clarify the notation, we point out that $\nabla W(u(t,x))$ is the gradient of $W$ evaluated at $u(t,x)$, while $D^2W(u)(\nu,\nu)$ stands for the quadratic form $\sum_{i,j=1}^m\frac{\partial^2 W(u)}{\partial u_i \partial u_j}\nu_i \nu_j$, $\forall u=(u_1,\ldots,u_m)\in\R^m$, $\forall \nu=(\nu_1,\ldots,\nu_m)\in\R^m$. We also denote respectively by $|\cdot|$ and $\cdot$, the Euclidean norm and inner product. Finally, given smooth maps $u:\R^2\to \R^m$, $u=(u_1,\ldots,u_m)$, and $\phi:\R^2\to\R^m$, $\phi=(\phi_1,\ldots,\phi_m)$, we set
$|\nabla u|^2:=\sum_{i=1}^m|\nabla u_i|^2$, and $\nabla u\cdot\nabla\phi:=\sum_{i=1}^m\nabla u_i\cdot \nabla \phi_i$.

System \eqref{system0} as well as the corresponding O.D.E.
\begin{equation}\label{ode0}
v''(x)=\nabla W(v(x)), \ v:\R\to\R^m\  (m\geq 2),\, x\in\R,
\end{equation}
have variational structure. We denote by
\begin{equation}\label{EE}
E_\Omega(u):=\int_\Omega\Big[\frac{1}{2}|\nabla u|^2+W(u)\Big], \ \Omega\subset \R^2,
\end{equation}
and
\begin{equation}\label{JJ}
J_I(v):=\int_I\Big[\frac{1}{2}| v'|^2+W(v)\Big], \ I\subset \R,
\end{equation}
the associated energy functionals.
We also recall that a heteroclinic orbit is a solution $e\in C^2(\R;\R^m)$ of \eqref{ode0} such that
$\lim_{x\to\pm\infty}e(x)=a^\pm$. A heteroclinic orbit is called \emph{minimal} if it is a minimizer of the Action functional \eqref{JJ}
in the class $A:=\{ v\in W_{\rm loc}^{1,2}(\R;\R^m):\ \lim_{x\to\pm\infty}v(x)=a^\pm\}$, i.e. if $J_\R(e)=\min_{v\in A} J_\R(v)=:J_{\mathrm{min}}$.
Assuming \eqref{w13}, we know that there exists at least one \emph{minimal}\footnote{Note that heteroclinic orbits are not always minimal: cf. \cite[Remark 3.6.]{antonop}.} heteroclinic orbit $e$ (cf. for instance \cite{antonop}, \cite{fusco2}, \cite{stern2} or \cite{alessio3}, for a general theorem about
the existence of heteroclinic connections). In addition, since the minima $a^\pm$ are nondegenerate, the convergence to the minima $a^\pm$ is exponential for every heteroclinic orbit $e$, i.e.
\begin{equation}\label{expest}
|e(x)-a^-|\leq K e^{kx}, \forall x\leq 0, \text{ and } |e(x)-a^+|\leq K e^{-kx}, \forall x\geq 0,
\end{equation}
where the constants $k,K>0$ depend on $e$ (cf. \cite[Proposition 6.5.]{antonop}.
Clearly, if $x\mapsto e(x)$ is a heteroclinic orbit, then the maps
\begin{equation}\label{transl}
x\mapsto e^T(x):=e(x-T),  \forall T\in\R,
\end{equation}
obtained by translating $x$, are still heteroclinic orbits.

\subsection{Heteroclinic orbits in Hilbert spaces}
In the first part of this paper, we establish the existence of minimal heteroclinic orbits in a Hilbert space $\mathcal H$, under very mild assumptions (cf. Theorem \ref{connh} below). Indeed, the potential $\mathcal W:\mathcal H\to [0,+\infty]$ is assumed to be weakly lower semicontinuous and to satisfy the standard asymptotic condition \eqref{liminfh}. For the sake of the applications to P.D.E. \eqref{system0}, we only consider the standard case of a double well potential $\mathcal W$ vanishing at $e^-$ and $e^+$, but clearly our approach can be applied to more general potentials vanishing either on finite sets or on manifolds (cf. \cite{antonop} in the finite dimensional case).
Denoting by $\langle \cdot,\cdot\rangle$ (resp. $\|\cdot\|$) the inner product (resp. the norm) in $\mathcal H$, the minimal heteroclinic $U$ will be obtained as a minimizer of the Action functional:
\begin{equation}\label{action}
\mathcal J_{\R}(V):=\int_\R \Big[\frac{1}{2}\|V'(t)\|^2+\mathcal W(V(t))\Big] \dd t ,
\end{equation}
in the constrained class $\mathcal{A}$ defined by:
\[\mathcal{A}=\Big\{V\in H_{\rm loc}^{1}(\R;\mathcal H):\left.  \begin{array}{l} \langle V(t)-e^-, \nn\rangle \leq 3l_0 /4,\text{ for }t\leq t_V^-,\\
 \langle V(t)-e^-, \nn\rangle \geq l_0 /4,\text{ for }t\geq t_V^+,\end{array}\right.\text{ for some }t_V^-<t_V^+\Big\},\]
where $\nn:=\frac{e^+-e^-}{\|e^+-e^-\|}$, and  $l_0:=\|e^+-e^-\|$. Note that in the definition of $\mathcal{A}$ no limitation is imposed on the numbers $t_V^-<t_V^+$ that may largely depend on $V$.
We refer to \cite{kreuter}, \cite{pap}, \cite{caz} and \cite{brezis2}, for the general theory of Sobolev spaces of vector-valued functions.

For nonsmooth potentials, the minimizer $U$ may be considered as a heteroclinic orbit in a generalized sense, since $U(t)$ converges weakly to $e^\pm$, as $t\to\pm\infty$ (cf. \eqref{weakcv}), and furthermore $U$ satisfies the equipartition relation \eqref{equipartition}. To illustrate Theorem \ref{connh} let us take for example $\mathcal W=\chi_{\mathcal H\setminus \{e^-,e^+\}}$, where $\chi$ is the characterictic function. Then, one obtains in view of \eqref{equipartition} that (up to translations):
\begin{equation}
U(t)=
\begin{cases}
e^-&\text{for  } t\leq 0  ,\\
e^-+\sqrt{2}t \nn &\text{for } 0\leq t\leq l_0/\sqrt{2},\\
e^+ &\text{for  } t\geq l_0/\sqrt{2}.
\end{cases}
\end{equation}
We refer for instance to \cite{caf}, \cite{savin} or \cite{alikakos2}, for the study of phase transition problems involving nonsmooth potentials.

In the case where $\mathcal W\in C^1(\mathcal H;\R)$ is smooth, the minimizer $U$ is a classical solution of the system
\begin{equation}\label{odeh}
	U''(t)=\nabla \mathcal W(U(t)),  \forall t \in  \R,
	\end{equation}
where given $u\in\mathcal H$, $\nabla \mathcal W(u)$ is the element of $\mathcal H$ corresponding to $D\mathcal W(u)\in \mathcal{H'}$ by identifying $\mathcal H$ with $\mathcal H'$ via the isomorphism:
\begin{equation}
\langle \nabla \mathcal W(u),v\rangle=D\mathcal W(u)v, \ \forall v\in\mathcal H.
\end{equation}

After these explanations, we give the complete statement of Theorem \ref{connh}:

\begin{theorem}\label{connh}
Let $\mathcal H$ be a Hilbert space\footnote{The existence of a minimizer $U$ satisfying \eqref{weakcv} and \eqref{equipartition} also holds if $\mathcal H$ is a reflexive Banach space.}, and assume that $\mathcal W:\mathcal H\to [0,+\infty]$ is a weakly lower semicontinuous function satisfying
\begin{equation}\label{doublew}
	\text{$\mathcal W$ has \emph{exactly} $2$ zeros $e^-$ and $e^+$,}
\end{equation}
and
\begin{equation}\label{liminfh}
\liminf_{\|v\|\to\infty}\mathcal W(v)>0.
\end{equation}
Then, the condition
\begin{equation}\label{bounden}
\inf_{V\in\mathcal{A}}\mathcal J_{\R}(V)<+\infty,
\end{equation}
implies that $\mathcal J_{\R}$ admits a minimizer $U\in\mathcal{A}$ i.e.
 $\mathcal J_{\R}(U)=\min_{V\in\mathcal{A}}\mathcal J_{\R}(V)$,
such that
\begin{subequations}
\begin{equation}\label{weakcv}
U(t)\rightharpoonup e^\pm, \text{ as } t\to\pm\infty,
\end{equation}
\begin{equation}\label{equipartition}
\text{$\frac{1}{2}\|U'(t)\|^2=\mathcal W(U(t))$ for a.e. $t\in\R$ (equipartition relation).}
\end{equation}
\end{subequations}
Assuming moreover that $\mathcal W\in C^1(\mathcal H;\R)$, then \eqref{bounden} holds and $U\in C^2(\R;\mathcal H)$ is a classical solution of \eqref{odeh}.
\end{theorem}

The method of constrained minimization to construct the minimal heteroclinic goes back to \cite{alikakos1}. However, most of the arguments used in finite dimensional spaces, fail in the infinite dimensional case due to the lack of compactness.
Thus, in order to recover compactness on closed balls, the idea is to work with the weak topology. On the other hand, the convergence in \eqref{weakcv} is established thanks to an argument first introduced in the context of fourth order O.D.E. (cf. \cite[Lemma 2.4.]{ps}). In what follows, we will see that for some specific potentials, the convergence to the minima $e^\pm$ may hold in the strong sense (cf. \eqref{lay1b}).

To apply Theorem \ref{connh} to P.D.E. problems, one may consider the solution  $\R\times\Omega\ni (t,x)\mapsto u(t,x)\in \R^m$ (with $\Omega\subset \R^n$) of a P.D.E., as a connecting orbit $t\mapsto U(t)\in\mathcal H$, $U(t):x\mapsto [U(t)](x):=u(t,x)$, taking its values in a Hilbert space $\mathcal H$ of functions, defined according to the boundary conditions satisfied by $u$. Of course, this can be done if the initial equation can be reduced to an O.D.E. similar to \eqref{odeh}, and if the boundary conditions satisfied by $u$ are appropriate.
The scope of this paper is to provide a method for performing such a reduction, and constructing various kinds of solutions of P.D.E. problems.

\subsection{First application: heteroclinic double layers}\label{ssec:appl1}
As a first application of Theorem \ref{connh} we give a new proof, in a slightly more general setting, of the existence of heteroclinic double layers (established by Schatzman \cite{scha}), since this result is particularly relevant for the phase transition system \eqref{system0}. Indeed, this construction provides for system \eqref{system0} the first examples of two-dimensional \emph{minimal solutions}, in the sense that
\begin{equation}\label{mineq}
E_{\mathrm{supp}\, \phi}(u)\leq E_{\mathrm{supp}\, \phi}(u+\phi), \ \forall \phi\in C^1_0(\R^2;\R^m).
\end{equation}
This notion of minimality is standard for many problems in which the energy of a localized solution is actually infinite due to non compactness of the domain.
Assuming that for system \eqref{system0}, with $W$ as in \eqref{w13}, there exist (up to translations) \emph{exactly} two minimal heteroclinic orbits $e^-$ and $e^+$ which are also nondegenerate\footnote{The heteroclinic orbits $e^\pm$ are nondegenerate in the sense that $0$ is a simple eigenvalue of the linearized operators $T:W^{2,2}(\R;\R^m)\rightarrow L^2(\R;\R^m)$, $T\varphi=-\varphi^{\prime\prime}+D^2W(e^\pm)\varphi$.},
Schatzman constructed a solution of \eqref{system0} such that
\begin{subequations}\label{layer}
\begin{equation}\label{lay1}
\forall x \in\R: \  \lim_{t\to\pm\infty}u(t,x)=e^\pm(x-m^\pm), \text{ for some constants $m^\pm\in\R$},
\end{equation}
\begin{equation}\label{lay2}
\forall t\in\R: \ \lim_{x\to\pm\infty}u(t,x)=a^\pm.
\end{equation}
\end{subequations}
Moreover, the convergence in \eqref{lay2} as well as in \eqref{lay1} is exponential, due to the nondegeneracy of $a^\pm$ and $e^\pm$.
This construction has initially been performed by Alama, Bronsard and Gui \cite{abg} for potentials $W$ invariant by the reflexion which exchanges $a^\pm$.
The symmetry assumption enabled the authors to control the translation parameters $m^\pm$, since they considered only solutions which were equivariant by the reflexion. In \cite{alessio}, the Alama-Bronsard-Gui solution was constructed under the weaker assumption \eqref{parti}, and the existence of an infinity of periodic solutions of \eqref{system0} was established (cf. also \cite{alessio2}).
Recently, new proofs of Schatzman's result were given in \cite{fusco} (where a Gibbon's type conjecture was also proved), and in \cite{monteil} via minimization of the Jacobi functional.

In Theorem \ref{connh2} below we obtain Schatzman's solution as a minimal heteroclinic orbit $U$ connecting $e^\pm$ in the appropriate Hilbert space. This construction highlights the real nature of the heteroclinic double layers, and provides a clear interpretation of the equipartition property \eqref{equi22} (already observed in the aforementioned works).
The boundary conditions \eqref{lay2} suggest to set
\begin{equation}
\ee_0(x)=
\begin{cases}
a^-,&\text{ for } x\leq -1,\\
a^-+(a^+ -a^-)\frac{x+1}{2},&\text{ for }-1 \leq x\leq 1,\\
a^+,&\text{ for } x\geq 1.
\end{cases}
\end{equation}
and work in the affine subspace\footnote{To stress the analogy with Theorem \ref{connh}, we denote again by $\mathcal H$, $\mathcal A$, $\mathcal W$, and $\mathcal J$, the Hilbert space, the constrained class, the potential, and the action functional, which are relevant in this subsection.} $\mathcal H:=\ee_0+L^2(\R;\R^m)=\{u=\ee_0+h: h\in  L^2(\R;\R^m)\}$ which has the structure of a Hilbert space with the inner product
\begin{equation}\label{innerp}
\langle u,v \rangle_{\mathcal H}:=\langle (u-\ee_0),(v-\ee_0)\rangle_{L^2(\R;\R^m)}, \ \forall u,v\in \mathcal H.
\end{equation}
We denote by $\|\cdot\|_{\mathcal H}$ the norm in $\mathcal H$, and by $d_{\mathcal H}(u,v):=\|u-v\|_{L^2(\R;\R^m)}$ the corresponding distance.
We shall also consider the Hilbert space $\mathcal{ \tilde H}:=\ee_0+H^1(\R;\R^m)=\{u=\ee_0+h: h\in  H^1(\R;\R^m)\}$ with the inner product
\begin{equation}\label{innerptilde}
\langle u,v \rangle_{\mathcal{ \tilde H}}:=\langle (u-\ee_0),(v-\ee_0)\rangle_{H^1(\R;\R^m)}, \ \forall u,v\in \mathcal{ \tilde H}.
\end{equation}
Similarly, $\|\cdot\|_{\mathcal{ \tilde H}}$, and $d_{ \mathcal{\tilde H}}(u,v):=\|u-v\|_{H^1(\R;\R^m)}$ stand for the norm and the distance in $\mathcal{\tilde H}$.
In view of \eqref{expest}, it is clear that $e\in \mathcal{ \tilde H}$, for every minimal heteroclinic $e$.

Next, we define in $\mathcal H$ the \emph{effective} potential $\mathcal W:\mathcal H \to [0,+\infty]$ by
\begin{equation}
\mathcal W(u)=
\begin{cases}
J_\R(u)-J_{\mathrm{min}}, &\text{ when the distributional derivative }  u' \in L^2(\R;\R^m),\\
+\infty,&\text{ otherwise,}
\end{cases}
\end{equation}
where $J_{\mathrm{min}}=\min_{v\in A} J_\R(v)$. 
Note that $\mathcal W\geq 0$, since $u' \in L^2(\R;\R^m)$ implies that $\lim_{x\to\pm\infty}u(x)=a^\pm$ i.e. $u \in A$, and thus $J_\R(u)\geq J_{\mathrm{min}}$. It is also obvious that $\mathcal W$ only vanishes on the set $F$ of minimal heteroclinics. More generally than in \cite{scha}, we assume that this set satisfies
\begin{equation}\label{parti}
F=F^-\cup F^+, \text{ with }  F^-\neq \varnothing, \  F^+\neq \varnothing, \text{ and } d_{\mathrm{min}}:=d_{\mathcal H}(F^-,F^+)>0
\end{equation}
(where $d_{\mathcal H}(F^-,F^+):=\inf\{ \|e^--e^+\|_{L^2(\R;\R^m)}: e^-\in F^-, e^+\in F^+\}$).
For instance, if $F$ contains (up to translations) a finite number of elements $e_1$,...,$e_N$, one may take $F^-=\{ x\mapsto e_1(x-T_1): T_1\in\R\}$, and $F^+=\{ x\mapsto e_k(x-T_k): T_k\in\R, k=2,\ldots,N\}$ (cf. \cite{abg} and \cite{scha}). In this case it is easy to check that $d_{\mathcal H}(F^-,F^+)>0$, since the map $\R\ni T\mapsto  e^T(x)=e(x-T)\in \mathcal H$ is continuous for every $e\in F$, and the images of two distinct minimal heteroclinics do not intersect. In Lemma \ref{lem2} below, we give explicit examples of potentials for which \eqref{parti} holds.

Finally we define the constrained class
\[\mathcal{A}=\Big\{V\in H_{\rm loc}^{1}(\R;\mathcal H):\left.  \begin{array}{l} d_{\mathcal H}(V(t),F^-) \leq d_{\mathrm{min}}/4,\text{ for }t\leq t_V^-,\\
d_{\mathcal H}(V(t),F^+) \leq d_{\mathrm{min}}/4,\text{ for }t\geq t_V^+,\end{array}\right.\text{ for some }t_V^-<t_V^+\Big\},\]
and the functional\footnote{In the proof of Theorem \ref{connh2}, it will appear how the energy functional $E$ of system \eqref{system0} is related to $\mathcal J$, and why the definition of $\mathcal W$ is relevant.}
\begin{equation}\label{action}
\mathcal J_{\R}(V):=\int_\R \Big[\frac{1}{2}\|V'(t)\|^2_{L^2(\R;\R^m)}+\mathcal W(V(t))\Big] \dd t .
\end{equation}
Since the effective potential $\mathcal W$ has been normalized by substracting the constant $ J_{\mathrm{min}}$ from $J_\R$, it follows that $\inf_{\mathcal A}\mathcal J_{\R}<\infty$. All variational constructions of the heteroclinic double layers are based on the minimization of this renormalized energy (cf. also \cite{alikakos1} for some other applications).
Proceeding as in Theorem \ref{connh} we are going to show that this solution is actually a minimizer of $\mathcal J_{\R}$ in $\mathcal A$:
\begin{theorem}\label{connh2}
Assume the potential $W$ satisfies \eqref{w13}, \eqref{parti}, and one of the following
\begin{subequations}
	\begin{equation}\label{eu1}
	\text{either there exists $\rho>0$ such that $W(su)\geq W(u)$ for $s \geq 1$ and $|u|=\rho$.}
	\end{equation}
	\begin{equation}\label{eu2}
	\text{or } \limsup_{|u|\to\infty} \frac{|\nabla W(u)|}{|u|^q}<\infty, \text{ for some $q\geq 2$.}
	\end{equation}
\end{subequations}
Then, $\mathcal J_{\R}$ admits a minimizer $U\in\mathcal{A}$ i.e. $\mathcal J_{\R}(U)=\min_{V\in\mathcal{A}}\mathcal J_{\R}(V)$, which is
	such that
	\begin{itemize}
		\item[(i)] $u\in C^2(\R^2;\R^m)$ where $u(t,x):=[U(t)](x)$, $t\mapsto U(t)\in \mathcal H$.
		\item[(ii)] $u$ solves \eqref{system0}
				together with the boundary conditions
		\begin{subequations}\label{layer}
			\begin{equation}\label{lay1b}
			\lim_{t\to\pm\infty}d_{\mathcal{ H}}(U(t),F^\pm)=0, 		\end{equation}
			\begin{equation}\label{lay2b}
			\lim_{x\to\pm\infty}u(t,x)=a^\pm \text{, uniformly when $t$ remains bounded}.
			\end{equation}
		\end{subequations}
				\item[(iii)] For every $t\in\R$, $u$ satisfies the equipartition relation $\frac{1}{2}\|U'(t)\|^2_{\mathcal H}=\mathcal W(U(t))$, or equivalently:
		\begin{equation}\label{equi22}
		\frac{1}{2}\int_\R|u_{t}(t,x)|^2\dd x=\int_\R \Big[\frac{1}{2}|u_x(t,x)|^2+ W(u(t,x))\Big] \dd x -J_{\mathrm{min}}.
		\end{equation}
		\item[(iv)] $u$ is a minimal solution of \eqref{system0} (cf. \eqref{mineq}).
				\end{itemize}
In addition, if \eqref{eu1} holds and $\mathcal W$ satisfies the nondegeneracy condition
\begin{equation}\label{nondegW}
\liminf_{d_{\mathcal H}(u,F)\to 0}\frac{\mathcal W(u)}{(d_{\mathcal H}(u,F))^2}>0,
\end{equation}
then there exist $e^\pm\in F^\pm$, and constants $k,K>0$ such that
	\begin{subequations}\label{layernew}
		\begin{equation}\label{lay1bnew}
\| U(t)-e^+\|_{\mathcal{\tilde H}}\leq K e^{-kt}, \, \forall t\geq 0, \text{ and } \|  U(t)-e^-\|_{\mathcal{\tilde H}}\leq K e^{kt}, \, \forall t\leq 0,	
		\end{equation}
		\begin{equation}\label{lay2bnew}
		| u(t,x)-a^+|\leq K e^{-kx}, \forall t\in\R,\forall x\geq 0, \text{ and }  | u(t,x)-a^-|\leq K e^{kx},  \forall t\in\R,\forall x\leq 0.
		\end{equation}
	\end{subequations}
\end{theorem}

To establish Theorem \ref{connh2}, the arguments in the proof of Theorem \ref{connh} need to be adjusted, since the set $F$ is unbounded. However, $\mathcal W$ and $F$ have nice properties, that allow us to address the lack of compactness issue. Indeed, $F$ intersected with closed balls of $\mathcal H$ is compact (cf. Lemma \ref{propcon} (i)), and $d_{\mathcal{\tilde H}}(u,F)\to 0$, as $\mathcal W(u)\to 0$ (cf. Lemma \ref{lem1w} (ii)).

Theorem \ref{connh2} outlines the hierarchical structure of solutions of \eqref{system0}, since by taking the limit of $u(t,x)$ as $t\to\pm\infty$ (resp. $x\to\pm\infty$), lower dimensional solutions are obtained. There is also a striking analogy between the functionals $J $ (cf. \eqref{JJ}) and $\mathcal J$. On the one hand, the zeros $a^\pm$ of $W$ (i.e. the global minimizers of $J$) have their counterparts in the minimal heteroclinics $e\in F$, which are the zeros of $\mathcal W$ (and the global minimizers of $\mathcal J$). On the other hand, the heteroclinic orbits of \eqref{ode0} (one dimensional solutions) have their counterparts in the heteroclinic orbit $U$ provided by Theorem \ref{connh2} which corresponds to a two dimensional solution of \eqref{system0}.

Finally, we point out that the shape of heteroclinics can be very complicated (cf. \cite{stern}), and that a nondegeneracy assumption similar to \eqref{nondegW} is needed to ensure the convergence of the orbit $U$ at $\pm\infty$, even in finite dimensional spaces (cf. \cite[Corollary 6.3.]{antonop}). The nondegeneracy assumption considered in \cite{scha} implies the existence of  $\alpha,\beta>0$ such that $d_{\mathcal{\tilde H}}(u,F)\leq\beta \Rightarrow \mathcal W(u)\geq \alpha (d_{\mathcal{\tilde H} }(u,F))^2$ (cf. \cite[Lemma 4.5.]{scha}). Clearly, this assumption is stronger than \eqref{nondegW}.

\subsection{Second application:}
In Theorem \ref{connh2} we constructed a heteroclinic orbit $U$ connecting at $\pm\infty$ the subsets $F^\pm$ in the Hilbert space $\mathcal{ H}$. Going further one may ask: what kind of solution is obtained if instead of $\mathcal H$, we consider another space? Assuming that $W$ satisfies \eqref{w13} as well as
\begin{equation}\label{partitilde}
F=F^-\cup F^+, \text{ with }  F^-\neq \varnothing, \  F^+\neq \varnothing, \text{ and } \tilde d_{\mathrm{min}}:=d_{\mathcal{\tilde H}}(F^-,F^+)>0,
\end{equation}
(cf. subsection \ref{ssec:appl1} for the definition of $\mathcal{\tilde H}$, $F$, and $\mathcal W$), we shall construct in this subsection a heteroclinic orbit $\tilde U$ connecting at $\pm\infty$ the subsets $F^\pm$ in $\mathcal{\tilde H}$.
This new orbit $\tilde U$ produces a heteroclinic double layers solution $\tilde u$ to the fourth order system
\begin{equation}\label{fourth}
\tilde u_{ttxx}(t,x)=\Delta \tilde u(t,x)-\nabla W(\tilde u(t,x)), \, \tilde u:\R^2\to\R^m.
\end{equation}
Proceeding as in Theorems \ref{connh} and \ref{connh2}, we shall establish that $\tilde U$ is a minimizer of the functional
\begin{equation}\label{actiontilde}
\mathcal{\tilde J}_{\R}(V):=\int_\R \Big[\frac{1}{2}\|V'(t)\|^2_{H^1(\R;\R^m)}+\mathcal W(V(t))\Big] \dd t .
\end{equation}
in the constrained class
\[\mathcal{\tilde A}=\Big\{V\in H_{\rm loc}^{1}(\R;\mathcal{\tilde H}):\left.  \begin{array}{l} d_{\mathcal{\tilde H}}(V(t),F^-) \leq \tilde d_{\mathrm{min}}/4,\text{ for }t\leq t_V^-,\\
d_{\mathcal{\tilde H}}(V(t),F^+) \leq \tilde d_{\mathrm{min}}/4,\text{ for }t\geq t_V^+,\end{array}\right.\text{ for some }t_V^-<t_V^+\Big\}.\]
\begin{theorem}\label{connh3}
Assume the potential $W$ satisfies \eqref{w13} and \eqref{partitilde}. Then, $\mathcal{ \tilde J}_{\R}$ admits a minimizer $\tilde U\in\mathcal{\tilde A}$ i.e. $\mathcal{\tilde J}_{\R}(\tilde U)=\min_{V\in\mathcal{\tilde A}}\mathcal{\tilde J}_{\R}(V)$, which is
	such that
	\begin{itemize}
		\item[(i)] $\tilde U\in C^2(\R;\mathcal{\tilde H})$ is a classical solution of system $\tilde U''(t)=\nabla \mathcal W(\tilde U(t))$, where $\mathcal W\in C^1(\mathcal{\tilde H};[0,\infty))$ (cf. Lemma \ref{lem1w} (iii)).
		\item[(ii)] Setting $\tilde u(t,x):=[\tilde U(t)](x)$, $t\mapsto \tilde U(t)\in \mathcal{\tilde H}$, we have $\tilde u \in H^1_{\mathrm{loc}}(\R^2;\R^m)$, $\tilde u_t, \, \tilde u_{tx}\in L^2(\R^2;\R^m)$, $\tilde u_x\in L^2((\alpha,\beta)\times \R;\R^m)$, $\forall [\alpha,\beta]\subset\R$, and $\tilde u$ is a weak solution of system \eqref{fourth}:
		\begin{equation}\label{wpde}
		\int_{\R^2} (\tilde u_{tx}\cdot \phi_{tx}+\nabla \tilde u\cdot\nabla \phi +\nabla W(\tilde u)\cdot \phi)=0, \ \forall \phi \in C^{2}_0(\R^2;\R^m),
		\end{equation}
 satisfying the boundary conditions
		\begin{subequations}\label{layertil}
			\begin{equation}\label{lay1btil}
			\lim_{t\to\pm\infty}d_{\mathcal{\tilde H}}(\tilde U(t),F^\pm)=0, 		\end{equation}
			\begin{equation}\label{lay2btil}
			\lim_{x\to\pm\infty}\tilde u(t,x)=a^\pm, \text{ uniformly when $t$ remains bounded}.
			\end{equation}
		\end{subequations}
				\item[(iii)] For every $t\in\R$, $\tilde u$ satisfies the equipartition relation $\frac{1}{2}\|\tilde U'(t)\|^2_{\mathcal{\tilde H}}=\mathcal W(\tilde U(t))$, or equivalently:
		\begin{equation}\label{equi22}
		\frac{1}{2}\int_\R(|\tilde u_{t}(t,x)|^2+|\tilde u_{tx}(t,x)|^2)\dd x=\int_\R \Big[\frac{1}{2}|\tilde u_x(t,x)|^2+ W(\tilde u(t,x))\Big] \dd x -J_{\mathrm{min}}.
		\end{equation}
		\item[(iv)] $u$ is a minimal solution of system \eqref{fourth} in the sense that
\begin{equation}\label{wmin}
\tilde E_{\mathrm{supp}\, \phi}(\tilde u )\leq \tilde E_{\mathrm{supp}\, \phi}(\tilde u+\phi ), \ \forall \phi\in C^2_0(\R^2;\R^m),
\end{equation}
where $\tilde E_\Omega(u):=\int_\Omega\big[\frac{1}{2}(|u_{tx}|^2+|\nabla u|^2)+W(u)\big]$ ($\Omega\subset \R^2$), is the energy functional associated to \eqref{fourth}.
				\end{itemize}
In addition, if  $\mathcal W$ satisfies the nondegeneracy condition
\begin{equation}\label{nondegWtil}
\liminf_{d_{\mathcal{\tilde H}}(u,F)\to 0}\frac{\mathcal W(u)}{(d_{\mathcal{\tilde H}}(u,F))^2}>0,
\end{equation}
then there exist $e^\pm\in F^\pm$, and constants $k,K>0$ such that
		\begin{equation}\label{lay1newtil}
		\| \tilde U(t)-e^+\|_{\mathcal{\tilde H}}\leq K e^{-kt}, \, \forall t\geq 0, \text{ and } \| \tilde U(t)-e^-\|_{\mathcal{\tilde H}}\leq K e^{kt}, \, \forall t\leq 0,		\end{equation}
	and the convergence in \eqref{lay2btil} is uniform for $t\in\R$.
\end{theorem}

\subsection{Other possible applications}
The previous method applies directly to construct heteroclinic double layers for the systems associated to the energy functionals  $E_\Omega(u)=\int_{\Omega}\big[  \big|\frac{\partial u}{\partial t}\big|^q +\big|\frac{\partial u}{\partial x}\big|^p+ W(u)\big]$, with $p,q\in(1,\infty)$, $u:\R^2\to\R^m$, $\Omega\subset \R^2$, and $W$ as in \eqref{w13}.
On the other hand, we expect that  Theorem \ref{connh} can be extended to fourth order systems by considering the functional $\mathcal J_{\R}(V)=\int_{\R} \big[\frac{1}{2}\|V''(t)\|^2+\mathcal W(V(t),V'(t))\big] \dd t $ (cf. \cite{ps} for the corresponding result in finite dimensional spaces). As a consequence, a heteroclinic double layers solution should be obtained for the system
\begin{equation*}
\Delta^2 u -\beta \Delta u +\nabla W(u)=0, \ u:\R^2\to\R^m, \ \beta\geq 0, \ W:\R^m\to [0,\infty),
\end{equation*}
which is called the extended Fisher-Kolmogorov equation. Finally, due to the variety of choices for the space $\mathcal H$, several types of boundary conditions may be considered in the applications of Theorem \ref{connh}.

\section{Proof of Theorem \ref{connh}}\label{sec:sec1}

We first notice that since $\mathcal W:\mathcal H\to[0,+\infty]$ is weakly lower semicontinuous, the function $t\mapsto \mathcal W(V(t))$ is lower semicontinuous (thus measurable), for every $V\in  W_{\rm loc}^{1,2}(\R;\mathcal H)$. Assumption \eqref{bounden} is satisfied for instance if $\mathcal W$ is bounded on the line segment $[e^-,e^+]$. Indeed, in this case the map $\VV_0\in\mathcal A$ defined by
\begin{equation}\label{minfun}
\VV_0(t)=\begin{cases}
e^-,&\text{ for } t\leq 0,\\
e^-+t(e^+ -e^-),&\text{ for }0 \leq t\leq 1,\\
e^+,&\text{ for } t\geq 1,
\end{cases}
 \end{equation}
is such that $\mathcal J_\R(\VV_0)<+\infty$. In what follows we assume that
$$\inf_{V\in\mathcal{A}}\mathcal J_{\R}(V)<\mathcal J_0, \text{ for a constant $\mathcal J_0 <+\infty$},$$ and we set
$\mathcal{A}_b=\{V\in\mathcal A: \mathcal J_{\R}(V)\leq \mathcal J_0\}$. It is clear that\[\inf_{V\in\mathcal{A}}\mathcal J_{\R}(V)=\inf_{V\in\mathcal{A}_b}\mathcal J_{\R}(V)<+\infty.\]
Our next claim is that finite energy orbits are equicontinuous and uniformly bounded:

\begin{lemma}\label{lem1}
There exist $M, M'>0$ such that $\sup_\R\|V(t)\|\leq M$, and $\|V(t_2)-V(t_1)\|\leq M' |t_2-t_1|^{1/2}$,
$\forall t_1, t_2\in\R$, $\forall V\in\mathcal{A}_b$. Moreover every map $V\in\mathcal A_b$ satisfies $V(t)\rightharpoonup e^\pm$, as $t\to\pm\infty$.
\end{lemma}
\begin{proof}
It is clear that for every $t_1<t_2$, and every $V\in\mathcal{A}_b$, we have
\begin{equation*}
\|V(t_2)-V(t_1)\|\leq \int_{t_1}^{t_2}\|V'(s)\|\dd s\leq \Big|\int_{t_1}^{t_2} \|V'(s)\|^2\dd s\Big|^{1/2}|t_2-t_1|^{1/2}\leq M'|t_2-t_1|^{1/2},
\end{equation*}
with $M'=\sqrt{2\mathcal J_0}$.
Next, in view of \eqref{liminfh}, $\|v\|\geq R$ implies that $\mathcal W(v)\geq m$ for some constant $m>0$, and $R > 0$ sufficiently large. Thus, for every $V\in\mathcal{A}_b$, we have $$m\mathcal L^1(\{ t\in \R: \|V(t)\|\geq R\})\leq \int_{\R}\mathcal W(V(t))\dd t\leq\mathcal J_0,$$ where $\mathcal L^1$ stands for the one dimensional Lebesgue measure. Assuming that $\|V(t)\|> R$, for some $t\in\R$, it follows that there exists $t_0<t$ such that $\|V(t_0)\|= R$, and $\|V(s)\|\geq R$, $\forall s\in [t_0,t]$. According to what precedes we can see that $m(t-t_0)\leq  \mathcal J_0$. Hence we deduce that
$\|V(t)-V(t_0)\|\leq M'(t-t_0)^{1/2}\leq \sqrt{2/m}\,\mathcal J_0$,
and $\|V(t)\|\leq R+\sqrt{2/m}\,\mathcal J_0=:M$.

Now, we recall that the ball $B_M:=\{v\in \mathcal H: \|v\| \leq M\}$ is compact for the weak topology. Let $\mathcal V=\{v\in\mathcal H: \langle f_j,v-e^+\rangle<2\delta, \forall j=1,\ldots, N\}$ (with $\delta>0$ and $f_j\in\mathcal H\setminus\{0\}$) be a neighbourhood of $e^+$ for the weak topology.
If we assume by contradiction the existence of a sequence $t_k$ such that $\lim_{k\to\infty} t_k=\infty$, and $V(t_k) \notin \mathcal V$ (i.e. $\langle f_{j_k},V(t_k)-e^+\rangle\geq 2\delta$ for some $j_k \in\{1,\ldots,N\}$), we get
$$\langle f_{j_k},V(t)-e^+\rangle \geq \langle f_{j_k},V(t)-V(t_k)\rangle+2\delta \geq  2\delta-M'\|f_{j_k}\| |t-t_k|^{1/2}\geq \delta,$$
provided that $|t-t_k|\leq\eta:=\min_{1\leq j\leq N}(\delta/M'\|f_j\|)^2$. Next, let $\mu$ be the infimum of $\mathcal W$ on the set
$$K_\delta:=\{ v\in B_M:    \langle v-e^-,\nn \rangle \geq l_0/4, \text{ and }  \langle f_j,v-e^+\rangle\geq \delta \text{ for some } j \in\{1,\ldots,N\}\},$$ which is compact for the weak topology. The weakly lower semicontinuity of $\mathcal W$ and \eqref{doublew}, imply that $\mu>0$, thus according to what precedes we have $\mathcal W(V(t))\geq \mu$, $\forall t\in [t_k-\eta,t_k+\eta]$, with $t_k\geq t_V^++\eta$. Finally, since the intervals $[t_k-\eta,t_k+\eta]$ may be assumed to be disjoint, we obtain $\mathcal J_\R(V)=\infty$, which is a contradiction. This establishes that $V(t)\rightharpoonup e^+$, as $t\to\infty$. Similarly we can prove that  $V(t)\rightharpoonup e^-$, as $t\to-\infty$.
\end{proof}

\begin{lemma}\label{lem1new}
Given a sequence $\{V_k\}\subset\mathcal A_b$, there exist a sequence $\{x_k\}\subset\R$, and a map $U\in\mathcal A_b$, such that $\mathcal J_\R(U)\leq \liminf_{k\to\infty} \mathcal J_\R(V_k)$, and up to subsequence the maps $\bar V_k(t):=V_k(t-x_k)$ satisfy
\begin{itemize}
\item[(i)] $\forall t\in\R$: $\bar V_k(t)\rightharpoonup U(t)$, as $k\to\infty$,
\item[(ii)] $\bar V'_k\rightharpoonup U'$ in $L^2(\R,\mathcal H)$, as $k\to\infty$.
\end{itemize}
\end{lemma}

\begin{proof}
By extracting if necessary a subsequence we may assume that $ \mathcal J_\R(V_k)$ converges to $\liminf_{k\to\infty} \mathcal J_\R(V_k)$, as $k\to \infty$.
For every $k$ we define the sequence $$-\infty<x_1(k)<x_2(k)<\ldots<x_{2N_k-1}(k)<x_{2N_k}(k)=\infty$$
by induction:
\begin{itemize}
\item $x_1(k)=\sup\{ t\in\R: \, \langle V_k(s)-e^-,\nn\rangle\leq 3l_0/4, \forall s\leq t\}<\infty$,
\item $x_{2i}(k)=\sup\{ t\in\R: \, \langle V_k(s)-e^-,\nn\rangle\geq l_0/4, \forall s\in [x_{2i-1}(k),t]\}\leq\infty$,
\item $x_{2i+1}(k)=\sup\{ t\in\R: \, \langle V_k(s)-e^-,\nn\rangle\leq 3l_0/4, \forall s\in [x_{2i}(k),t]\}<\infty$, if $x_{2i}(k)<\infty$,
\end{itemize}
where $i=1,\ldots,N_k$.
In addition, we set
\begin{itemize}
\item $y_{2i-1}(k)=\sup\{ t\leq x_{2i-1}(k): \, \langle V_k(t)-e^-,\nn\rangle\leq l_0/4\}$,
\item $y_{2i}(k)=\sup\{ t\leq x_{2i}(k): \, \langle V_k(t)-e^-,\nn\rangle\geq 3l_0/4\}$, if $x_{2i}(k)<\infty$.
\end{itemize}

\begin{figure}[h]
\begin{center}
\includegraphics[scale=1]{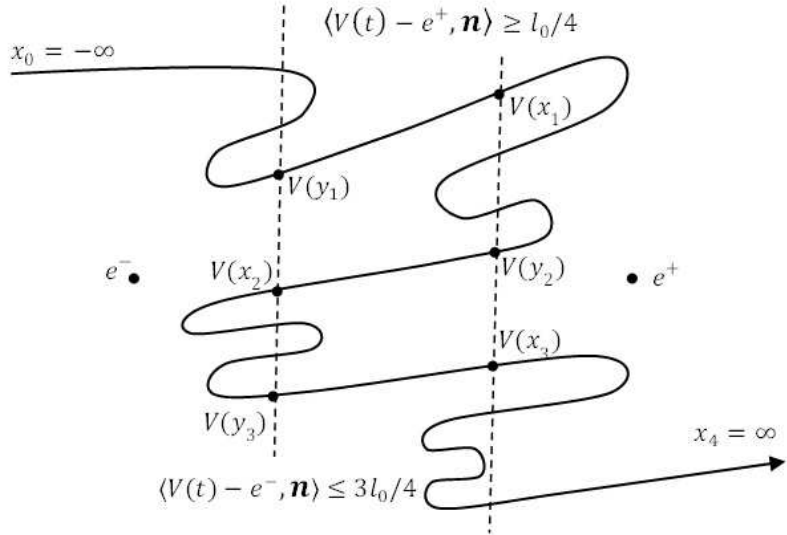}
\end{center}
\caption{The sequence $-\infty=x_0<y_1<x_1\leq y_2<x_2<\ldots<x_{2N}=\infty$, ($N=2$).}
\label{fig}
\end{figure}

Next, we notice that the set $K:=\{ v\in\mathcal H: \|v\| \leq M, l_0/4\leq\langle v-e^-,\nn\rangle\leq 3l_0/4 \}$ is compact for the weak topology. As a consequence of \eqref{doublew} and the lower semicontinuity of $\mathcal W$, we have $\mathcal W_0:=\min_{v\in K}\mathcal W(v)=\mathcal W(v_0)$, for some $v_0\in K$, thus $\mathcal W_0>0$. Finally, since
$$\mathcal J_{[y_j(k),x_j(k)]}(V_k)\geq \int_{y_j(k)}^{x_j(k)}\sqrt{2\mathcal W(V_k(t))}\|V'_k(t)\|\dd t\geq \sqrt{\mathcal W_0/2}\,l_0,$$
holds for every $k\geq 1$ and $j=1,\ldots, 2N_k-1$, we can see that $(2N_k-1)\sqrt{\mathcal W_0/2}\,l_0\leq \mathcal J_0$, i.e.
the integers $N_k$ are uniformly bounded. By passing to a subsequence, we may assume that $N_k$ is a constant integer $N\geq 1$.

Our next claim (cf. \cite[Lemma 2.4.]{ps}) is that up to subsequence, there exist an integer $i_0$ ($1\leq i_0\leq N$) and an integer $j_0$ ($i_0\leq j_0\leq N$) such that
\begin{itemize}
\item[(a)] the sequence $x_{2j_0-1}(k)-x_{2i_0-1}(k)$ is bounded,
\item[(b)] $\lim_{k\to\infty}(x_{2i_0-1}(k)-x_{2i_0-2}(k))=\infty$,
\item[(c)] $\lim_{k\to\infty}(x_{2j_0}(k)-x_{2j_0-1}(k))=\infty$,
\end{itemize}
where for convenience we have set $x_0(k):=-\infty$.

Indeed, we are going to prove by induction on $N\geq 1$, that given $2N+1$ sequences $-\infty\leq x_0(k)<x_1(k)<\ldots<x_{2N}(k)\leq\infty$, such that $\lim_{k\to\infty}(x_1(k)-x_0(k))=\infty$, and $\lim_{k\to\infty}(x_{2N}(k)-x_{2N-1}(k))=\infty$, then up to subsequence the properties (a), (b), and (c) above hold, for two fixed indices $1\leq i_0\leq j_0\leq N$.
When $N=1$, the assumption holds by taking $i_0=j_0=1$. Assume now that $N>1$, and let $l \geq 1$ be the largest integer such that
the sequence $x_{l}(k)-x_{1}(k)$ is bounded. Note that $l<2N$. If $l$ is odd, we are done, since the sequence $x_{l+1}(k)-x_{l}(k)$ is unbounded, and thus we can extract a subsequence $\{n_k\}$ such that $\lim_{k\to\infty}(x_{l+1}(n_k)-x_{l}(n_k))=\infty$. Otherwise $l=2m$ (with $1\leq m<N$), and the sequence $x_{2m+1}(k)-x_{2m}(k)$ is unbounded. We extract a subsequence $\{n_k\}$ such that $\lim_{k\to\infty}(x_{2m+1}(n_k)-x_{2m}(n_k))=\infty$.
Then, we apply the inductive statement with $N'=N-m$, to the $2N'+1$ sequences $x_{2m}(n_k)<x_{2m+1}(n_k)<\ldots<x_{2N}(n_k)$.

At this stage, we consider appropriate translations of the sequence $\{V_k\}$, by setting $\bar V_k(t)=V_k(t-x_{2i_0-1}(k))$. Since $\{\bar V'_k\}$ is uniformly bounded in $L^2(\R,\mathcal H)$, it follows that up to subsequence
$\bar V'_k\rightharpoonup V$ in $L^2(\R,\mathcal H)$, and  moreover
\begin{equation}\label{la1}
\int_{\R}\| V\|^2\leq\liminf_{k\to\infty} \int_{\R}\|\bar V'_k\|^2.
\end{equation}
On the other hand, we write $\bar V_k(t)=\bar V_k(0)+\int_0^t\bar V'_k(s)\dd s$,
and notice that up to subsequence $\bar V_k(0)\rightharpoonup u_0$ in $\mathcal H$, since $\|\bar V_k(0)\|\leq M$ (cf. Lemma \ref{lem1}). Our claim is that
$U(t):=u_0+\int_0^tV(s)\dd s$ has all the desired properties. Indeed, since $\int_0^t\bar V'_k(s)\dd s\rightharpoonup  \int_0^t V(s)\dd s$ holds in $\mathcal H$ for every $t\in\R$, we also have $\bar V_k(t)\rightharpoonup U(t)$ for every $t\in\R$. In view of the weakly lower semicontinuity of $\mathcal W$, this implies that $\liminf_{k\to\infty}\mathcal W(\bar V_k(t))\geq \mathcal W(U(t))$ for every $t\in\R$, thus by Fatou's Lemma we obtain
\begin{equation}\label{la2}
\int_{\R}\mathcal W(U(t))\dd t\leq\liminf_{k\to\infty} \int_{\R}\mathcal W(\bar V_k(t))\dd t.
\end{equation}
Combining \eqref{la1}\ with \eqref{la2} it is clear that $\mathcal J_\R(U)\leq\liminf_{k\to\infty}\mathcal J_{\R}(V_k)$. To conclude it remains to show that $U\in\mathcal A$. In view of the above property (b) it follows that $\langle U(t)-e^-,\nn\rangle\leq 3l_0/4$, for every $t\leq 0$. Similarly, in view of (a) and (c), we have $\langle U(t)-e^-,\nn\rangle\geq l_0/4$, for $t\geq T>0$ large enough.
\end{proof}

Applying Lemma \ref{lem1new} to a minimizing sequence i.e. $\{V_k\}\subset\mathcal A_b$ such that $$\lim_{k\to\infty}\mathcal J_\R(V_k)=\inf_{V\in\mathcal{A}_b}\mathcal J_{\R}(V),$$ we immediately obtain the existence of the minimizer $U$.
To show that the minimizer $U$ satisfies the equipartition property (ii) we are going to check that
\begin{equation}\label{check1}
0=\int_\R \Big(\frac{1}{2}\|U'(t)\|^2-\mathcal W(U(t))\Big)\phi(t)\dd t,\  \forall \phi \in C^\infty_0(\R;\R).
\end{equation}
Actually, since every $\phi \in C^\infty_0(\R;\R)$ is the uniform limit of step functions, we just need to prove that
\begin{equation}\label{check1}
\int_a^b \frac{1}{2}\|U'(t)\|^2=\int_a^b\mathcal W(U(t))\dd t, \ \forall [a,b]\subset \R.
\end{equation}
For every $\kappa>0$, let
\begin{equation*}
V_\kappa(t)=\begin{cases}
U(t),&\text{ for } t\leq a,\\
U(a+\frac{t-a}{\kappa}),&\text{ for } t\in [a, a+\kappa( b-a)],\\
U(t+(1-\kappa)(b-a)),&\text{ for } t\geq a+\kappa (b-a).
\end{cases}
 \end{equation*}
It is easy to see that $V_\kappa\in \mathcal A$ and,
\begin{equation}\label{compk}
\mathcal J_\R(V_\kappa)-\mathcal J_\R(U)=\int_a^b \frac{(1-\kappa)}{2\kappa}\|U'(t)\|^2+(\kappa-1)\int_a^b\mathcal W(U(t))\dd t.
\end{equation}
Since $\mathcal J_\R(V_\kappa)-\mathcal J_\R(U)\geq 0$ by the minimality of $U$, letting $\kappa \to 1^+$ and $\kappa \to 1^-$ in \eqref{compk}, we obtain \eqref{check1}.

Finally we assume that $\mathcal W\in C^1(\mathcal H;\R)$. Given $\xi   \in C^\infty_0(\mathcal H;\R)$, and $\lambda\in\R$, we compute
$$\frac{\dd}{\dd \lambda}\Big|_{\lambda=0}\mathcal J_\R(U+\lambda \xi)=\int_\R\big[\langle U'(t),\xi'(t)\rangle+\langle\nabla\mathcal W(U(t)),\xi(t)\rangle\big]\dd t.$$
By the minimality of $U$, we have $\mathcal J_\R(U+\lambda \xi)-\mathcal J_\R(U)\geq 0$, hence
$$\int_\R\big[\langle U'(t),\xi'(t)\rangle+\langle\nabla\mathcal W(U(t)),\xi(t)\rangle\big]\dd t=0.$$
This implies that the derivative of
$t\mapsto U'(t)$ in
$\mathcal D'(\R;\mathcal H)$
is $t\mapsto \nabla\mathcal W(U(t))$
and that $U\in C^2(\R;\mathcal H)$ is a classical solution of \eqref{odeh}.

\section{Properties of the effective potential $\mathcal W$ and of the set of minimal heteroclinics $F$}

We establish below some properties of the effective potential $\mathcal W$ defined in subsection \ref{ssec:appl1}, assuming that the function $W$ satisfies \eqref{w13}:

\begin{lemma}\label{lem1w}
\begin{itemize}
\item[(i)] The potential $\mathcal W$ is sequentially weakly lower semicontinuous.
\item[(ii)] Let $\{u_k\}\subset \mathcal H$ be such that $\lim_{k\to\infty}\mathcal W(u_k)=0$.
Then, there exist a sequence $\{x_k\}\subset\R$, and $e\in F$, such that (up to subsequence) the maps $\bar u_k(x):=u_k(x-x_k)$ satisfy $\lim_{k\to\infty}\|\bar u_k-e\|_{H^{1}(\R;\R^m)}=0$.
As a consequence, $d_{\mathcal{\tilde H}}(u,F)\to 0$, as $\mathcal W(u)\to0$, and for every $c_1>0$, there exists $c_2>0$ such that $d_{\mathcal{H}}(u,F)\geq c_1 \text{(resp. $d_{\mathcal{\tilde H}}(u,F)\geq c_1$)} \Rightarrow \mathcal W(u)\geq c_2$.
\item[(iii)] $\mathcal W$ restricted to $\mathcal{\tilde H}$ is a $C^1(\mathcal{\tilde H};[0,\infty))$ smooth function, and $D\mathcal W(u)h=\int_\R [u'\cdot h'+ \nabla W(u)\cdot h]$, $\forall u\in\mathcal{\tilde H}$, $\forall h\in H^1(\R;\R^m)$.
\end{itemize}
\end{lemma}

\begin{proof}
(i) Let $\{u_k\}\subset \mathcal H$ be such that $u_k \rightharpoonup u$ in $\mathcal H$ (i.e.
$u_k -u\rightharpoonup 0$ in $L^2(\R;\R^m)$), and let us assume that
$l=\liminf_{k\to \infty}\mathcal W(u_k)<\infty$ (since otherwise the statement is trivial).
By extracting a subsequence we may assume that $\lim_{k\to \infty}\mathcal W(u_k)=l$. In view of Lemma \ref{lem1} (applied in the finite dimensional case with $W$ instead of $\mathcal W$), the sequence $\{u_k\}$ is equicontinuous and uniformly bounded. Thus, the theorem of Ascoli implies that $u_k \to \tilde u$ in $C_{\mathrm{loc}}(\R;\R^m)$, as $k\to\infty$ (up to subsequence).
On the other hand, since $\|u'_k\|_{L^2(\R;\R^m)}$ is bounded, we have that $u'_k \rightharpoonup v$, in $L^2(\R;\R^m)$ (up to subsequence). In addition, one can easily see that  $u=\tilde u\in H_{\rm loc}^{1}(\R;\R^m)$, and $u'=v$. Finally, by the weakly semicontinuity of the $L^2(\R;\R^m)$ norm and Fatou's Lemma (cf. the end of the proof of Lemma \ref{lem1new}), we deduce that $\mathcal W(u)\leq l$, i.e.
$\mathcal W(u)\leq \liminf_{k\to \infty}\mathcal W(u_k)$.

(ii) We first establish that given $u\in\mathcal H$ such that $u' \in L^2(\R;\R^m)$, and $e \in F$, we have
\begin{equation}\label{formula}
\mathcal W(u)=\int_\R \Big[\frac{1}{2}|u'-e'|^2 +W(u)-W(e)-\nabla W(e)\cdot (u-e)\Big].
\end{equation}
In view of \eqref{expest}, it is clear that $e''=\nabla W(e) \in L^2(\R;\R^m)$, thus $e'\in H^{1}(\R;\R^m)$. As a consequence, we can see that $\int_\R e''\cdot(u-e)=-\int_\R e'\cdot (u'-e')$, and
\begin{align*}
\mathcal W(u)&=\int_\R \Big[\frac{1}{2}|u'|^2-\frac{1}{2}|e'|^2 +W(u)-W(e)\Big]\\
&=\int_\R \Big[\frac{1}{2}|u'|^2-\frac{1}{2}|e'|^2-e'\cdot (u'-e')  +W(u)-W(e)- e''\cdot(u-e)\Big],
\end{align*}
from which \eqref{formula} follows.

Now, we consider a sequence $\{u_k\}\subset \mathcal H$ such that $\lim_{k\to\infty}\mathcal W(u_k)=0$. According to Lemma \ref{lem1new}, there exist a sequence $\{x_k\}\subset\R$, and $e\in F$, such that (up to subsequence) the maps $\bar u_k(x):=u_k(x-x_k)$ satisfy
\begin{equation}\label{propaa}
 \lim_{k\to\infty}\bar u_k(x)= e(x), \, \forall x\in\R.
\end{equation}
Having a closer look at the proof of Lemma \ref{lem1new}, we can show that in the case of a finite dimensional space, the convergence in \eqref{propaa} actually holds in $C_{\mathrm{loc}}(\R;\R^m)$\footnote{Indeed, when $\mathcal H=\R^m$, one can apply in the proof of Lemma \ref{lem1new}  the theorem of Ascoli to the sequence $\bar V_k$, since by Lemma \ref{lem1} it is equicontinuous and uniformly bounded.}.

Our claim is that
\begin{equation}\label{claime}
\lim_{k\to\infty}\|\bar u_k-e\|_{H^{1}(\R;\R^m)}=0.
\end{equation}
According to hypothesis \eqref{w2} we have
\begin{subequations}
\begin{equation}\label{convex1}
	W(u)\geq\frac{c}{2}|u-a^\pm|^2 , \forall u: |u-a^\pm|\leq r,
	\end{equation}
\begin{equation}\label{convex2}
	W(v)-W(u)-\nabla W(u)\cdot (v-u)\geq \frac{c}{2}|v-u|^2, \forall u,v: |u-a^\pm|\leq r, \ |v-a^\pm|\leq r.
	\end{equation}
\end{subequations}
Let $\mu>0$ be such that
\begin{equation}\label{convex3}
	W(u)\leq\frac{\mu}{2}|u-a^\pm|^2 , \forall u\in\R^m: |u-a^\pm|\leq r,
	\end{equation}
let $\epsilon\in (0,r)$, and let $\nu$ be a unit vector of $\R^m$. We notice using \eqref{convex3} that the map $[0,1]\ni x\mapsto z(x)=a^\pm+\epsilon \nu x$, is such that $J_{[0,1]}(z)\leq \frac{\mu+1}{2}\epsilon^2$. As a consequence,
\begin{equation}\label{bound4}
	\inf\{J_{[\alpha,\beta]}(v): v\in H^{1}([\alpha,\beta];\R^m), |v(\alpha)-a^-|= \epsilon, |v(\beta)-a^+|=\epsilon\}\geq J_{\mathrm{min}}-(\mu+1)\epsilon^2,
	\end{equation}
since otherwise we can construct a map in $A$ whose action is less than $J_{\mathrm{min}}$. On the other hand we have
\begin{equation}\label{bound5}
	\inf\{J_{[\alpha,\beta]}(v): v\in H^{1}([\alpha,\beta];\R^m), |v(\alpha)-a^\pm|= \epsilon, |v(\beta)-a^\pm|= r\}\geq \sqrt{c}(r-\epsilon)\epsilon.
	\end{equation}
Indeed, for such a map $v$, we can check that $$J_{[\alpha,\beta]}(v)\geq \int_\alpha^\beta \sqrt{2W(v)}|v'|\geq \sqrt{c}(r-\epsilon)\epsilon.$$
Let $\epsilon_0\in(0,r)$ be such that $(\mu+2)\epsilon^2<\sqrt{c}(r-\epsilon)\epsilon$, $\forall \epsilon <\epsilon_0$. Next, for $\epsilon<\epsilon_0$ fixed, choose an interval $[\lambda^-,\lambda^+]$ such that $|e(x)-a^-|\leq \epsilon/2$, $\forall x\leq \lambda^-$, and $|e(x)-a^+|\leq \epsilon/2$, $\forall x\geq \lambda^+$. According to \eqref{propaa}, we have for $k\geq N$ large enough:
\begin{subequations}
\begin{equation}\label{bound6}
	|\bar u_k(\lambda^\pm)-a^\pm|< \epsilon,
	\end{equation}
\begin{equation}\label{bound7}
	\Big| \int_{[\lambda^-,\lambda^+]} ( W(\bar u_k)-W(e)-\nabla W(e)\cdot (\bar u_k-e))\Big|< \epsilon^2,
	\end{equation}
\begin{equation}\label{bound7b}
	\|\bar u_k-e\|_{L^2( [\lambda^-,\lambda^+];\R^m)}< \epsilon,
	\end{equation}
\begin{equation}\label{bound8}
	\mathcal W(\bar u_k)<\epsilon^2.
	\end{equation}
\end{subequations}
Then, combining \eqref{bound4} with \eqref{bound8}, one can see that
\begin{equation}\label{bound9}
	J_{\R\setminus [\lambda^-,\lambda^+]}(\bar u_k)< (\mu+2)\epsilon^2< \sqrt{c}(r-\epsilon)\epsilon.
	\end{equation}
Therefore, in view of \eqref{bound5} and \eqref{bound6}, it follows that $|\bar u_k(x)-a^-|\leq r$ (resp. $|\bar u_k(x)-a^+|\leq r$), $\forall x\leq \lambda^-$ (resp.  $\forall x\geq \lambda^+$). Furthermore, as a consequence of \eqref{convex2} we get
\begin{equation}\label{bound10}
	\int_{\R\setminus [\lambda^-,\lambda^+]}	(W(\bar u_k)-W(e)-\nabla W(e)\cdot (\bar u_k-e))\geq \frac{c}{2} \|\bar u_k-e\|_{L^2( \R\setminus[\lambda^-,\lambda^+];\R^m)}^2.
	\end{equation}
To conclude, we apply formula \eqref{formula} to $\bar u_k$, and combine \eqref{bound8} with \eqref{bound7} and \eqref{bound10}, to obtain
\begin{equation}\label{bound11}
	\|\bar u_k-e\|_{L^2( \R\setminus [\lambda^-,\lambda^+];\R^m)}< \frac{2\epsilon}{\sqrt{c}}, \text{  and  } \|\bar u'_k-e'\|_{L^2( \R;\R^m)}<2\epsilon.
	\end{equation}
 Finally, in view of \eqref{bound7b}, we have $\|\bar u_k-e\|_{L^2( \R;\R^m)}<\big(1+ \frac{2}{\sqrt{c}})\epsilon$. This establishes our claim \eqref{claime}, from which the statement (ii) of Lemma \ref{lem1w} is straightforward.

(iii) We recall that $\sigma:=\sup_{e\in F}\|e\|_{L^\infty(\R;\R^m)}<\infty$ (cf. Lemma \ref{lem1}).
Given $u\in \mathcal{\tilde H}$, set $\kappa_1:=\max(\|u\|_{L^\infty(\R;\R^m)},\sigma)$, and $\kappa_2:=\sup\{ |
D^2W(v)(\nu,\nu)|: |v|\leq 2\kappa_1, |\nu|=1\}$.
From formula \eqref{formula}, it is clear that $$\mathcal W(u)\leq \frac{1}{2}\|u'-e'\|_{L^2(\R;\R^m)}^2+\frac{\kappa_2}{2} \|u-e\|_{L^2(\R;\R^m)}^2<\infty.$$ On the other hand, one can see that
$\nabla W(u)\in L^2(\R;\R^m)$. Furthermore, when $\|h\|_{H^1(\R;\R^m)}$ is small enough, such that $\|h\|_{L^\infty(\R;\R^m)}<\kappa_1$, we have
\begin{align*}
\Big|\mathcal W(u+h)-\mathcal W(u)-\int_\R[u'\cdot h'+\nabla W(u)\cdot h]\Big|\leq \frac{1}{2}\|h'\|_{L^2(\R;\R^m)}^2+\frac{\kappa_2}{2}\|h\|_{L^2(\R;\R^m)}^2.
\end{align*}
This proves that $\mathcal W$ is differentiable at $u$, and $D\mathcal W(u)h=\int_\R[u'\cdot h'+\nabla W(u)\cdot h]$.
\end{proof}

From the arguments in the proof of Lemma \ref{lem1w}, we deduce some useful properties of the set $F$ (defined in subsection \ref{ssec:appl1}).
\begin{lemma}\label{propcon}
\begin{itemize}
\item[(i)] Let $\{e_k\}\subset F$ be bounded in $\mathcal H$, then
there exists $e\in F$, such that up to subsequence $\lim_{k\to\infty}\|e_k-e\|_{H^{1}(\R;\R^m)}=0$.
\item[(ii)] There exists a constant $\gamma>0$, such that for every $e\in F$, we can find $T \in \R$ such that setting $e^T (x)=e(x-T)$, we have $\|e^T\|_{\mathcal{\tilde H}}\leq \gamma$.
\item[(iii)] For every $v\in \mathcal H$ (resp. $v\in \mathcal{\tilde H}$), there exists $e\in F$ such that $d_{\mathcal H}(v,F)=\|v-e\|_{\mathcal H}$ (resp. $d_{\mathcal{\tilde H}}(v,F)=\|v-e\|_{\mathcal{\tilde H}})$.
\end{itemize}
\end{lemma}
\begin{proof}
(i) Since $\{e_k\}\subset F$ is bounded in $\mathcal H$, we have up to subsequence $e_k \rightharpoonup e$ in $\mathcal H$, as $k\to\infty$, for some $e\in\mathcal H$.
Proceeding as in the proof of Lemma \ref{lem1w} (i), we first obtain that (up to subsequence) $e_k \to  e$ in $C_{\mathrm{loc}}(\R;\R^m)$, as $k\to\infty$, with $e\in F$. Next, we reproduce the arguments after \eqref{claime}, with $e_k$ instead of $\bar u_k$.

(ii) Assume by contradiction the existence of a sequence $\N \ni k \mapsto e_k\in F$, such that $\|e^T_k\|_{\mathcal{\tilde H}}\geq k$, $\forall T\in \R$. Then, by Lemma \ref{lem1w} (ii), there exists a sequence $\{x_k\}\subset\R$, and $e\in F$, such that (up to subsequence) the maps $e_k^{x_k}$ satisfy $\lim_{k\to\infty}\| e_k^{x_k}-e\|_{\mathcal{\tilde H}}=0$. Clearly, this is a contradiction.

(iii) Let $\{e_k\}\subset F$ be a sequence such that $\|v-e_k\|_{\mathcal H}\leq d_{\mathcal H}(v,F)+\frac{1}{k}$, $\forall k$. Then, in view of (i) we have (up to subsequence) $e_k\to e$ in $\mathcal H$, as $k\to\infty$, with $e\in F$. As a consequence
$d_{\mathcal H}(v,F)=\|v-e\|_{\mathcal H}$.
\end{proof}

In Lemma \ref{lem2} below, we give examples of potentials for which assumption \eqref{parti} holds.
\begin{lemma}\label{lem2}
Let $W\in C^2(\R^2;\R) $ be a potential satisfying \eqref{w13}. In addition we assume that
\begin{itemize}
\item $W(u_1,u_2)=W(u_1,-u_2)$,
\item $a^\pm=(\pm \lambda,0)$,
\item the heteroclinic orbit $\eta$ taking its values onto the open line segment $(a^-,a^+)$ is not minimal.\footnote{An explicit example of a potential satisfying all the above assumptions is constructed in \cite[Remark 3.6.]{antonop}.}
\end{itemize}
Then, $F$ is partitioned into two nonempty sets $F^\pm$, such that $d_{\mathcal H}(F^-,F^+)>0$.
 \end{lemma}
\begin{proof}
By symmetry, if $x \mapsto (e_1(x),e_2(x))\in\R^2$ is a minimal heteroclinic orbit, then $x\mapsto(e_1(x),-e_2(x))$ is also a minimal heteroclinic orbit.
Since the images of two distinct minimal heteroclinic orbits do not intersect, and the heteroclinic orbit $\eta$ is not minimal, it follows that a minimal heteroclinic orbit either takes its values in the upper half-plane $\{u_2>0\}$ or in the lower half-plane $\{u_2<0\}$. We denote by $F^\pm$ the corresponding subsets. If $d_{\mathcal H}(F^-,F^+)=0$, then there exists a sequence $e_k=(f_k,g_k)\subset F^+$ such that $\lim_{k\to\infty}\|g_k\|_{L^2(\R)}=0$.
According to Lemma \ref{lem1new}, there also exists a sequence $x_k\in \R$, such that $\lim_{k\to\infty}e_k(x-x_k)=(f(x),0)=:u(x)\in A$. Furthermore, we have $J_\R(u)\leq J_{\mathrm{min}}$. Therefore, $u$ is a minimal heteroclinic orbit coinciding up to translations with $\eta$. This is a contradiction, since the orbit $\eta$ is not minimal.
\end{proof}

\section{Proof of Theorem \ref{connh2}}

\begin{proof}[Existence of the minimizer $U$]

To see that $\inf_{V\in\mathcal{A}}\mathcal J_{\R}(V)<\infty$, we take $\VV_0 \in \mathcal A$ as in \eqref{minfun}, with $e^\pm\in F^\pm$. Since $e^-$ and $e^+$ satisfy the exponential estimate \eqref{expest}, it is clear that
$\mathcal J_{\R}(\VV_0)<\infty$. Next, we define the constants
\begin{itemize}
\item $\mathcal W_1:=\inf\{\mathcal W (v): d_{\mathcal H}(v,F)\in [d_{\mathrm{min}}/8,d_{\mathrm{min}}/4]\}\in(0,\infty)$ (cf. Lemma \ref{lem1w} (ii)),
\item $M>0$ such that $\mathcal W(v)\leq 1\Rightarrow \|v\|_{L^\infty(\R;\R^m)}\leq M$ (cf. Lemma \ref{lem1} applied to $W$),
\item $C>0$ such that $|D^2W(v)(\nu,\nu)|\leq C$, $\forall v$: $|v|\leq M$, $\forall \nu \in \R^m$: $|\nu|=1$,
\item $\eta \in (0, d_{\mathrm{min}}/8)$ such that $(1+C)\eta^2<\sqrt{2\mathcal W_1}(d_{\mathrm{min}}/8)$,
\item $\mathcal W_2:=\inf\{\mathcal W(v): d_{\mathcal H}(v,F)\geq \eta\}\in(0,\infty)$ (cf. Lemma \ref{lem1w} (ii)),
\item $\epsilon\in(0,1)$ such that $\epsilon<\sqrt{2\mathcal W_1}(d_{\mathrm{min}}/8)-(1+C)\eta^2$,
\end{itemize}
and consider a minimizing sequence i.e. $\{V_k\}\subset\mathcal A$ such that $\lim_{k\to\infty}\mathcal J_\R(V_k)=\inf_{V\in\mathcal{A}}\mathcal J_{\R}(V)$. For every $k$, we set
$$\lambda^-_k:=\sup S_k^-,\text{ where } S_k^-:= \{ t\in\R: \mathcal W(V_k(t))\leq\epsilon, \ d_{\mathcal H}(V_k(t),F^-)\leq \eta\},$$
$$\lambda^+_k:=\inf S_k^+,\text{ where } S_k^+:=\{ t\geq \lambda_k^-: \mathcal W(V_k(t))\leq\epsilon, \ d_{\mathcal H}(V_k(t),F^+)\leq \eta\}.$$
Note that $S_k^\pm\neq\varnothing$, since $\mathcal J_{\R}(V_k)<\infty$ implies that $\liminf_{|t|\to\infty}\mathcal W(V_k(t))=0$, and also $\liminf_{|t|\to\infty}d_{\mathcal H}(V_k(t),F)=0$ by Lemma \ref{lem1w} (ii).
Moreover, one can see that actually $\lambda^-_k=\max S_k^-$, and $\lambda^+_k=\min S_k^-$.
Indeed, let $\{t_j\}\subset S_k^-$ be a sequence such that $t_j\to\lambda_k^-$, as $j\to\infty$. Then, there exists a sequence $\{e_j\}\subset F^-$ such that $\|V_k(t_j)-e_j\|_{\mathcal H}\leq \eta$. In addition, in view of Lemma \ref{propcon} (i), we have up to subsequence $e_j\to e$ in $\mathcal H$, as $j\to \infty$, for some $e\in F^-$, thus $\|V_k(\lambda_k^-)-e\|_{\mathcal H}\leq \eta$. On the other hand, from Lemma \ref{lem1w} (i) we get immediately that $\mathcal W(V_k(\lambda_k^-))\leq \epsilon$.

By definition of $\lambda^\pm_k$, either $\mathcal W(V_k(t))>\epsilon$ or $d_{\mathcal H}(V_k(t),F)>\eta$ holds for $t\in(\lambda^-_k,\lambda^+_k)$. Thus, we have $\mathcal W(V_k(t))\geq\min(\epsilon,\mathcal W_2)$, $\forall t\in(\lambda^-_k,\lambda^+_k)$, and as a consequence of the boundedness of the sequence $k\mapsto \mathcal J_\R(V_k)$, it follows that $\Lambda:=\sup_k(\lambda_k^+-\lambda_k^-)\in (0,\infty)$.
Our next claim is that we may assume that the minimizing sequence $\{V_k\}$ satisfies (cf. \cite[Lemma 4.3.]{antonop}):
\begin{equation}\label{solinas}
d_{\mathcal H}(V_k(t),F^-)\leq d_{\mathrm{min}}/4, \forall t\leq \lambda_k^-, \text{ and } d_{\mathcal H}(V_k(t),F^+)\leq d_{\mathrm{min}}/4, \forall t\geq \lambda_k^+.
\end{equation}
Indeed, if a map $V_k$ is such that for instance $d_{\mathcal H}(V_k(t_0),F^-)> d_{\mathrm{min}}/4$, for some $t_0< \lambda_k^-$, we can construct a competitor $\tilde V_k\in\mathcal A$, such that $\mathcal J_\R(\tilde V_k)\leq \mathcal J_\R(V_k)$, and   \eqref{solinas} holds for $\tilde V_k$. To see this, let $e^-\in F^-$ be such that $\|V_k(\lambda_k^-)-e^-\|_{\mathcal H}=d_{\mathcal H}(V_k(\lambda_k^-),F^-)\leq \eta$, and set
\begin{equation}\label{competitor}
\tilde V_k(t):=
\begin{cases}
V_k(t),&\text{ for } t\geq \lambda_k^-,\\
e^-+(t-\lambda^-_k+1)(V_k(\lambda_k^-)-e^-),&\text{ for } t\in [\lambda_k^--1,\lambda_k^-]\\
e^-,&\text{ for } t\leq \lambda_k^--1.
\end{cases}
 \end{equation}
One can see that $\int_{-\infty}^{\lambda_k^-}\|\tilde V'_k\|^2_{\mathcal H}=\|V_k(\lambda_k^-)-e^-\|^2_{\mathcal H}$. Next, applying formula \eqref{formula} to $e=e^-$ and $u=V_k(\lambda_k^-)$ together with $\mathcal W(V_k(\lambda_k^-))\leq\epsilon$, we obtain
$\int_\R\frac{1}{2}|(V_k(\lambda_k^-)-e^-)'|^2\leq \epsilon+\frac{C}{2}\|V_k(\lambda_k^-)-e^-\|^2_{\mathcal H}$. Finally, a second application of formula \eqref{formula} to $e^-$ and $e^-+s(V_k(\lambda_k^-)-e^-)$, with $s\in[0,1]$, gives
$\mathcal W(\tilde V_k(t))\leq \epsilon+C\|V_k(\lambda_k^-)-e^-\|^2_{\mathcal H}$, $\forall  t\in [\lambda_k^--1,\lambda_k^-]$. Thus we have checked that $\mathcal J_{(-\infty,\lambda_k^-]}(\tilde V_k)\leq \epsilon+(C+1)\|V_k(\lambda_k^-)-e^-\|^2_{\mathcal H}\leq \epsilon+(C+1)\eta^2$.
On the other hand, assuming that
$d_{\mathcal H}(V_k(t_0),F^-)> d_{\mathrm{min}}/4$ holds for some $t_0< \lambda_k^-$, we have
$$\mathcal J_{[t_0,\lambda_k^-]}(V_k)\geq\int_{[t_0,\lambda_k^-]}\sqrt{2\mathcal  W(V_k)}\|V'_k\|_{\mathcal H} \geq \sqrt{2\mathcal W_1 }( d_{\mathrm{min}}/8).$$
Therefore, by definition of $\epsilon$ and $\eta$ we deduce that $\mathcal J_{(-\infty,\lambda_k^-]}(\tilde V_k)\leq \mathcal J_{(-\infty,\lambda_k^-]}( V_k)$. This proves our claim \eqref{solinas}.

To show the existence of the minimizer $U$, we shall consider appropriate translations of the sequence $v_k(t,x):=[V_k(t)](x)$ ($\R\ni t\mapsto V_k(t)\in \mathcal H$), with respect to the variables $x$ and $t$. Then, we shall establish the convergence of the translated maps to the minimizer $U$. Given $T\in\R$, and $V\in \mathcal H=\ee_0+L^2(\R;\R^m)$, we denote by $L^T(V)$ the map of $\mathcal H$ defined by $\R\ni x\mapsto V(x-T)\in \R^m$. It is obvious that $\mathcal W(L^T(V))=\mathcal W(V)$.
Similarly, if $t\mapsto V(t)$ belongs to $H^{1}_{\mathrm{loc}}(\R;\mathcal H)$, we obtain that $t\mapsto L^T(V(t))$ also belongs to $H^{1}_{\mathrm{loc}}(\R;\mathcal H)$, with $\|(L^T V)'(t)\|_{L^2(\R;\R^m)}=\|V'(t)\|_{L^2(\R;\R^m)}$.

In view of Lemma \ref{propcon} (ii), for every $k$, we can find $T_k\in \R$ and $e_k \in F^-$ such that
$\|e_k\|_{\mathcal H}\leq \gamma$ and $\||L^{T_k}V_k(\lambda_k^-)-e_k\|_{\mathcal H}\leq \eta$. We set $\bar V_k(t):=L^{T_k}(V_k(t+\lambda_k^-))$. Clearly, $\bar V_k \in H^{1}_{\mathrm{loc}}(\R;\mathcal H)$ satisfies $\mathcal J_\R(\bar V_k)=\mathcal J_\R(V_k)$, as well as
\begin{equation}\label{solinas2}
d_{\mathcal H}(\bar V_k(t),F^-)\leq d_{\mathrm{min}}/4, \forall t\leq 0, \text{ and } d_{\mathcal H}(\bar V_k(t),F^+)\leq d_{\mathrm{min}}/4, \forall t\geq \Lambda.
\end{equation}
Since $\|\bar V_k(0)\|_{\mathcal H}\leq \eta+\gamma$ holds for every $ k$, 
we have that (up to subsequence) $\bar V_k(0) \rightharpoonup  u_0$ in $\mathcal H$, as $k\to\infty$, for some $u_0\in\mathcal H$.
Next, proceeding as in the proof of Lemma \ref{lem1new}
we can see that (up to subsequence)
$\bar V'_k\rightharpoonup V$ in $L^2(\R;L^2(\R;\R^m))$ as $k\to\infty$, and  moreover
setting $\bar V_k(t)=\bar V_k(0)+\int_0^t\bar V'_k(s)\dd s$, and $U(t)= u_0+\int_0^t V(s)\dd s$, we have $\bar V_k(t)\rightharpoonup U(t)$ in $\mathcal H$, as $k\to \infty$, $\forall t\in\R$.
The fact that $\mathcal J_\R(U)\leq\liminf_{k\to\infty}\mathcal J_{\R}(V_k)$ follows as in the proof of Lemma \ref{lem1new} from the sequentially weakly lower semicontinuity of $\mathcal W$ (cf. Lemma \ref{lem1w} (i)). To conclude that
$\mathcal J_{\R}(U)=\min_{V\in\mathcal{A}}\mathcal J_{\R}(V)$, we are going to check that $U$ satisfies \eqref{solinas2}. Indeed, given $t\leq0$, let $\{e_k\}\subset F^-$ be such that $\|\bar V_k(t)-e_k\|_{\mathcal H}\leq d_{\mathrm{min}}/4$, $\forall k$. Since $\{e_k\}$ is bounded in $\mathcal H$, we have (up to subsequence) $\lim_{k\to\infty} e_k=e$ in $\mathcal H$, for some $e\in F^-$ (cf. Lemma \ref{propcon} (i)). Thus, it is clear that $d_{\mathcal H}(U(t),F^-)\leq d_{\mathcal H}(U(t),e)\leq \liminf_{k\to\infty}\|\bar V_k(t)-e_k\|_{\mathcal H}\leq d_{\mathrm{min}}/4$. Similarly, $d_{\mathcal H}(U(t),F^+)\leq d_{\mathrm{min}}/4$ holds for $t\geq\Lambda$.
\end{proof}

\begin{proof}[Proof of (i), (ii), (iii) and (iv)] We first establish two lemmas:
\begin{lemma}\label{sobolev}
Writing $U(t)=\ee_0+H(t)$, with $$H\in H^1_{\mathrm{loc}}(\R;L^2(\R;\R^m))\subset L^2_{\mathrm{loc}}(\R;L^2(\R;\R^m)),$$ and identifying $H$ with a $L^2_{\mathrm{loc}}(\R^2;\R^m)$ function via $h(t,x):=[H(t)](x)$, we have
\begin{itemize}
\item[(i)] $h \in H^1_{\mathrm{loc}}(\R^2;\R^m)$, $h_t\in L^2(\R^2;\R^m)$,
\item[(ii)] and $\|h_x\|_{L^2((\alpha,\beta)\times \R;\R^m)}^2\leq C_0 (|\beta-\alpha|)$, for a constant $C_0>0$ depending only on the length of the interval $(\alpha,\beta)\subset\R$.
\end{itemize}
\end{lemma}
\begin{proof}
We recall that given any interval $(\alpha,\beta)$, we can identify $L^2((\alpha,\beta)\times \R;\R^m)$ with $L^2((\alpha,\beta);L^2(\R;\R^m))$ via the canonical isomorphism
\begin{align*}\label{isom}
L^2((\alpha,\beta)\times \R;\R^m)&\simeq L^2((\alpha,\beta);L^2(\R;\R^m))\\
f&\simeq [(\alpha,\beta)\ni t \mapsto [F(t)]:x\mapsto f(t,x)], \ F(t)\in L^2(\R;\R^m).
\end{align*}
Let $g(t,x):=[U'(t)](x)$, with $g\in L^2(\R^2;\R^m)$, and let us prove that $h_t=g$.
Given a function $\phi \in C^\infty_0(\R^2;\R^m)$, we also view it as a map $\Phi\in C^1(\R; L^2(\R;\R^m))$, $t\mapsto \Phi(t)$, by setting $[\Phi(t)](x):=\phi(t,x)$.
Assuming that $\supp\Phi\subset(\alpha,\beta)$, we have $$\int_{\R^2} [h\phi_t+g\phi]=\int_\alpha^\beta(\langle H(t),\Phi_t(t)\rangle_{\mathcal H}+\langle H_t(t),\Phi(t)\rangle_{\mathcal H})\dd t,$$
and clearly the second integral vanishes if $H\in C^1([\alpha,\beta]; L^2(\R;\R^m))$. Since $H$ can be approximated in $H^1((\alpha,\beta); L^2(\R;\R^m))$ by $C^1([\alpha,\beta]; L^2(\R;\R^m))$ maps, we deduce that $\int_{\R^2} [h\phi_t+g\phi]=0$, i.e. $h_t=g$.

On the other hand, $\int_\R\mathcal W(U(t))\dd t<\infty$ implies that for a.e. $t\in\R$, we have $\mathcal W(U(t))<\infty$, and $U(t)\in \mathcal{\tilde H}$.
By using difference quotients, we can see that
\begin{equation}\label{sob2}
\int_\R\Big|\frac{h(t,x+\eta)-h(t,x)}{\eta}\Big|^2\dd x \leq k\int_\R |h_x|^2 \leq 4k(\mathcal W(U(t))+J_{\mathrm{min}}) +2k\|\ee'_0\|_{L^2(\R;\R^m)}^2,
\end{equation}
holds for a.e. $t\in\R$, for $\eta\in\R\setminus\{0\}$, and some constant $k>0$. Thus, the difference quotients $\frac{h(t,x+\eta)-h(t,x)}{\eta}$ are bounded in $L^2((\alpha,\beta)\times\R;\R^m)$ for every interval $[\alpha,\beta]\subset\R$, and as a consequence $h_x\in L^2((\alpha,\beta)\times\R;\R^m)$. Finally, an integration of \eqref{sob2} gives $\|h_x\|_{L^2((\alpha,\beta)\times \R;\R^m)}^2\leq C_0 (|\beta-\alpha|)$, with $$C_0=4k\int_\R\mathcal W(U(t))\dd t+2k|\beta-\alpha|(2J_{\mathrm{min}}+\|\ee'_0\|_{L^2(\R;\R^m)}^2).$$
\end{proof}

\begin{lemma}\label{bU}
If \eqref{eu1} holds, there exists a minimizer $ U$ of $\mathcal J_\R$ in $\mathcal A$ satisfying:
\begin{equation}\label{boundU}
\|U(t)\|_{L^\infty(\R;\R^m)}\leq \rho, \, \forall t \in \R.
\end{equation}
\end{lemma}
\begin{proof}
Let $P:\R^m\to\R^m$ be the projection onto the closed ball $\{u\in \R^m: |u|\leq \rho\}$. Given $V\in \mathcal H$, it is clear that the map $PV:x\mapsto P(V(x))$ belongs to $\mathcal H$. In addition, given
$V_1,V_2\in \mathcal H$, we have $\|PV_1-PV_2\|_{ \mathcal H}\leq \|V_1-V_2\|_{ \mathcal H}$. As a consequence, the map $PU: t\mapsto P(U(t))\in \mathcal H$ belongs to $H^1_{\mathrm{loc}}(\R;\mathcal H)$, and $\|(PU)'(t)\|_{L^2(\R;\R^m)}\leq \|U'(t)\|_{L^2(\R;\R^m)}$ holds for a.e. $t\in \R$. On the other hand, it is clear that $\mathcal W((PU)(t))\leq \mathcal W(U(t))$ holds for every $t\in\R$. To deduce that $PU$ is a minimizer of $\mathcal J_\R$ in $\mathcal A$, it remains to check that $PU$ satisfies \eqref{solinas2}. Given $t\leq0$, let $e\in F^-$ be such that $\| U(t)-e\|_{\mathcal H}\leq d_{\mathrm{min}}/4$, and note that $\|e\|_{L^\infty(\R;\R^m)}\leq\rho$, since $e$ is a minimal heteroclinic. This implies that for every $x\in \R$, we have $|[PU(t)](x)-e(x)|\leq |[U(t)](x)-e(x)|$. Thus, it follows that $d_{\mathcal H}(PU(t),e)\leq d_{\mathcal H}(U(t),e)\leq d_{\mathrm{min}}/4$. Similarly, $d_{\mathcal H}(PU(t),F^+)\leq d_{\mathrm{min}}/4$ holds for $t\geq\Lambda$.
\end{proof}

Given a function $\phi \in C^1_0(\R^2;\R^m)$, we also view it as a map $\Phi\in C^1_0(\R; L^2(\R;\R^m))$, $t\mapsto \Phi(t)$, by setting $[\Phi(t)](x):=\phi(t,x)$.
For every $\lambda \in \R$, it is clear that
\begin{equation}\label{var1}
\mathcal J_\R(U)\leq\mathcal J_\R(U+\lambda\Phi),
\end{equation}
and
\begin{equation}\label{var33}
\frac{\dd}{\dd\lambda}\Big|_{\lambda=0}\int_\R\frac{1}{2}\|U'(t)+\lambda \Phi'(t)\|_{L^2(\R;\R^m)}^2\dd t=\int_\R\langle U'(t),\Phi'(t)\rangle_{L^2(\R;\R^m)}\dd t.
\end{equation}
On the other hand, since $\int_\R\mathcal W(U(t))\dd t<\infty$, it follows that for a.e. $t\in\R$, we have $\mathcal W(U(t))<\infty$, and $U(t)\in \mathcal{\tilde H}$.
Our claim is that
\begin{equation}\label{var22}
\frac{\dd}{\dd\lambda}\Big|_{\lambda=0}\int_\R\mathcal W(U(t)+\lambda\Phi(t))\dd t=\int_{\R}\psi(t)\dd t,
\end{equation}
where $\psi(t):=\int_\R\big[\frac{\dd [U(t)]}{\dd x}\cdot \frac{\partial\phi(t,x)}{\partial x}+\nabla W([U(t)](x))\cdot\phi(t,x)\big]\dd x$.
We first notice that for every $\lambda\neq 0$, the functions
$\psi_\lambda(t):=\frac{1}{\lambda}[\mathcal W(U(t)+\lambda \Phi(t))-\mathcal W(U(t))]$ are defined a.e. Moreover, we can see that $\psi_\lambda(t)$ is equal to
\begin{equation}\label{cvboundd}
\int_\R\Big[\frac{\dd [U(t)]}{\dd x}\cdot \frac{\partial\phi(t,x)}{\partial x}+\frac{\lambda}{2}\Big|\frac{\partial\phi(t,x)}{\partial x}\Big|^2+\nabla W([U(t)](x)+c_\lambda(t,x)\lambda \phi(t,x))\cdot\phi(t,x)\Big]\dd x,
\end{equation}
with $0\leq c_\lambda(t,x)\leq 1$. As a consequence, we obtain $\lim_{\lambda\to 0}\psi_\lambda(t)=\psi(t)$ for a.e. $t\in\R$. Finally, setting $u(t,x):=[U(t)](x)$, we have $u\in H^1_{\mathrm{loc}}(\R^2;\R^m)\subset L^q_{\mathrm{loc}}(\R^2;\R^m)$, $\forall q\geq 2$ (cf. Lemma \ref{sobolev}), and moreover $u\in L^\infty(\R^2;\R^m)$ when \eqref{eu1} holds (cf. \eqref{boundU}). This implies that either under assumption \eqref{eu2} or \eqref{eu1}, we can find a function $\Psi\in L^1(\R)$ such that
$|\psi_\lambda(t)|\leq \Psi(t)$ holds a.e., when $|\lambda|$ is small. Thus, we deduce \eqref{var22} by dominated convergence. Now, we gather the previous results to conclude.
In view of \eqref{var1}, \eqref{var33} and \eqref{var22}, the minimizer $U$ satisfies the Euler-Lagrange equation
\begin{equation}\label{euler1}
\int_{\R}(\langle U'(t),\Phi'(t)\rangle_{L^2(\R;\R^m)}+\psi(t))\dd t=0.
\end{equation}
 which is equivalent to
\begin{equation}\label{euler2}
\int_{\R^2}(\nabla u\cdot \nabla\phi +\nabla W(u)\cdot \phi)=0.
\end{equation}
By elliptic regularity (cf. respectively \cite[Theorem 8.34. and Corollary 4.14.]{gilbarg} under assumption \eqref{eu1}, and \cite[Theorem 8.8. and Corollary 4.14.]{gilbarg} under assumption \eqref{eu2}) it follows that $u$ is a classical solution of \eqref{system0}. When \eqref{eu1} holds it is clear that $u$ is uniformly continuous on $\R^2$, since $|\nabla u|$ is bounded on $\R^2$.
Similarly, when \eqref{eu2} holds, Lemma \ref{sobolev} implies that $\|u\|_{H^1(D;\R^m)}$ and $\|\nabla W(u)\|_{L^2(D;\R^m)}$ are uniformly bounded on the discs $D$ of radius $1$ included in the strip $[\alpha,\beta]\times\R$ (with $[\alpha,\beta]\subset\R$). Thus, in view of  \cite[Theorem 8.8.]{gilbarg}, $u$ is uniformly continuous on the strip $[\alpha,\beta]\times\R$.
To prove \eqref{lay2b}, assume by contradiction the existence of a sequence $(t_k,x_k)$ such that $\lim_{k\to\infty}x_k=\infty$, $t_k\in[\alpha,\beta]$, and $|u(t_k,x_k)-a^+|>\epsilon>0$. As a consequence of the uniform continuity of $u$, we can construct a sequence of disjoint discs of fixed radius, centered at $(t_k,x_k)$, over which $W(u)$ is bounded uniformly away from zero. This clearly violates the finiteness of $E_{[\alpha,\beta]\times\R}(u)=\mathcal J_{[\alpha,\beta]}(U)+J_{\mathrm{min}}(\beta-\alpha)$. To prove \eqref{lay1b}, assume by contradiction the existence of a sequence $t_k$ such that $\lim_{k\to\infty}t_k=\infty$, and $d_{\mathcal H}(U(t_k),F^+)>2\epsilon>0$. Since $\R\ni t\mapsto U(t)\in \mathcal H$ is uniformly continuous, we can construct a sequence of disjoint intervals $
[t_k-\eta, t_k+\eta]$ of fixed length over which $d_{\mathcal H}(U(t),F^+)>\epsilon>0$, and $\mathcal W(U(t))$ is bounded uniformly away from zero (cf. Lemma \ref{lem1w} (ii)). This again violates the finiteness of $\mathcal J_{\R}(U)$. Finally, the equipartition property (iii) is established as in Theorem \ref{connh}, and (iv) follows from \eqref{var1}, since $E_{[\alpha,\beta]\times \R}(u+\phi)=\mathcal J_{[\alpha,\beta]}(U+\Phi)+(\beta-\alpha)J_{\mathrm{min}}$, if $\supp \phi\subset (\alpha,\beta)\times\R$.
\end{proof}

\begin{proof}[Proof of \eqref{layernew}]
The proof proceeds as in \cite[Proposition 6.1.]{antonop}. In view of \eqref{nondegW}, let $t_0\in\R$ and $\kappa>0$ be such that
\begin{equation}
\label{tube}
\mathcal W(U(t))\geq \kappa d_{\mathcal H}^2(U(t),F^-),\ \forall t\leq t_0.
\end{equation}
For $t \leq t_0$ fixed, let $e^- \in  F^-$ be such that $d_{\mathcal H}(U(t),e^-)=d_{\mathcal H}(U(t),F^-)$, and define the map
\begin{equation}\label{comparison3}
Z(s)=\begin{cases}
U(t)+(t-s)(e^--U(t)),&\text{ for } t-1\leq s \leq t,\\
e^-, &\text{ for } s \leq t-1.
\end{cases}\end{equation}
By reproducing the argument after \eqref{competitor} we obtain
$\mathcal J_{[t-1,t]}(Z)\leq  \mathcal W(U(t))+(C+1)d_{\mathcal H}^2(U(t),F^-)$, with $C=\sup_{|u|\leq \rho, |\nu|=1}|D^2W(u)(\nu,\nu)|$.
Thanks to the variational characterization of $U$ and to \eqref{tube}, it follows that
\begin{multline}\label{comput1}
\kappa \int_{-\infty}^t d^2_{\mathcal H}(U(s), F^-) \dd s \leq\int_{-\infty}^t \mathcal W(U(s)) \dd s \leq
\mathcal J_{(-\infty,t]}(U)
\leq\mathcal J_{[t-1,t]}(Z)\leq   \mathcal W(U(t))+(C+1)d_{\mathcal H}^2(U(t),F^-),
\end{multline}
Setting $\theta(t):=\int_{-\infty}^t (d^2_{\mathcal H}(U(s), F^-)+\mathcal W(U(s)) )\dd s$, we deduce that $\theta \in W^{1,1}_{\mathrm{loc}}((-\infty,t_0])$, and
$\gamma \theta\leq \theta'$ holds a.e. on $(-\infty,t_0]$ for some constant $\gamma>0$. By integrating this inequality, it follows that
\begin{equation}\label{comput0}
\theta(t)\leq \theta(t_0)e^{\gamma(t-t_0)}.
\end{equation}
Now, we notice that by the equipartition property, we have
\begin{equation}\label{intU2}
\int_{-\infty}^t \| U'(s)\|_{L^2(\R:\R^m)}^2\dd s\leq 2\theta(t_0)e^{\gamma(t-t_0)}, \, \forall t\leq t_0,
\end{equation}
and for every $ j\in \N$:
\begin{equation}\label{intU2b}
\int_{t-j-1}^{t-j} \| U'(s)\|_{L^2(\R:\R^m)}\dd s\leq \Big(\int_{t-j-1}^{t-j} \| U'(s)\|_{L^2(\R:\R^m)}^2\dd s\Big)^{\frac{1}{2}}\leq\sqrt{ 2\theta(t_0)}e^{\frac{\gamma}{2}(t-t_0)}e^{-\frac{\gamma}{2}j}.
\end{equation}
Therefore,
\begin{equation}\label{sum}
\int_{-\infty}^{t} \| U'(s)\|_{L^2(\R:\R^m)}\dd s\leq \frac{\sqrt{ 2\theta(t_0)}}{1-e^{-\frac{\gamma}{2}}}e^{\frac{\gamma}{2}(t-t_0)}<\infty,
\end{equation}
and $U(t)\to e^-$ in $\mathcal H$, as $t\to-\infty$, for some $e^-\in F^-$. Similarly, we establish the existence of $e^+\in F^+$ such that  $U(t)\to e^+$ in $\mathcal H$, as $t\to\infty$.

Next, we choose $\epsilon\in(0,r/2)$ such that $(\mu+1)\epsilon^2<\sqrt{c}(r-2\epsilon)\epsilon$, where $\mu$ is defined in \eqref{convex3}.
Let $L>0$ be such that $|e^\pm(x)-a^-|<\epsilon/4$ (resp. $|e^\pm(x)-a^+|<\epsilon/4$) holds for every $x\leq -L$ (resp. $x\geq L$).
Our claim is that $|u(t,x)-a^-|\leq r$ (resp. $|u(t,x)-a^+|\leq r$) holds for $x\leq- L-1$ (resp. $x\geq L+1$) and $|t|\geq T$ large enough.
Without loss of generality we are only going to check that $|u(t,x)-a^-|\leq r$ holds for $x\leq- L-1$ and $t\geq T$ large enough. Indeed, otherwise there exists a sequence $(t_k,x_k)$ such that $\lim_{k\to\infty}t_k=\infty$, $x_k\leq -L-1$, and $|u(t_k,x_k)-a^-|> r$. Up to subsequence, we have $\lim_{k\to\infty}u(t_k,x)=e^+(x)$ for a.e. $x\in\R$. Let $T>0$ be such that $|u(t_k,L_0)-a^-|\leq \epsilon/2$ holds for some $L_0 \in(-L-1,-L)$, when $t_k\geq T$.
By the uniform continuity of $u$, there exists $\eta>0$ (independent of k) such that $|u(t,x_k)-a^-|\geq r-\epsilon/2$ and $|u(t,L_0)-a^-|\leq \epsilon$ hold for $t\in[t_k-\eta,t_k+\eta]$. In view of \eqref{bound4} and \eqref{bound5} we deduce that
$\mathcal W(U(t))\geq  \sqrt{c}(r-2\epsilon)\epsilon-(\mu+1)\epsilon^2>0$, $\forall t\in[t_k-\eta,t_k+\eta]$, with $t_k\geq T$. Thus we obtain $\int_{\R} \mathcal W(U(t))\dd t=\infty$ which is a contradiction. This establishes our claim, and now \eqref{lay2bnew} follows easily from a standard comparison argument. Moreover, using elliptic estimates we also obtain that $|\nabla u(t,x)|\leq K' e^{-k'|x|}$ holds for some constants $k',K'>0$, and $|D^2 u|$ is bounded on $\R^2$. As a consequence, the function $\R\ni t\mapsto \psi(t):=\mathcal W(U(t))$ is Lipschitz, since $\psi'(t)=\int_\R[ u_x(t,x)\cdot u_{tx}(t,x)+\nabla W(u(t,x))\cdot u_t(t,x)]\dd x$ is uniformly bounded by a constant $\beta>0$. We infer that
\begin{equation}\label{infer}
\mathcal W(U(t))\leq 2\sqrt{\beta\theta(t_0)}e^{\frac{\gamma}{2}(t-t_0)}, \, \forall t\leq t_0.
\end{equation}
To see this, let $t\leq t_0$ be fixed and let $\lambda:=\psi(t)$.
For $s\in [t-\frac{\lambda}{2\beta},t]$, we have $\psi(s)\geq \psi(t)-\beta|s-t|\geq\frac{\lambda}{2}$. Thus, we get
$\frac{\lambda^2}{4\beta}\leq\int_{t-\frac{\lambda}{2\beta}}^t \psi(s)\dd s\leq \theta(t_0)e^{\gamma(t-t_0)}$, from which \eqref{infer} is straightforward. Finally, \eqref{sum} implies that
\begin{equation}\label{cvv1}
\| U(t)-e^-\|_{L^2(\R:\R^m)}\leq\frac{\sqrt{ 2\theta(t_0)}}{1-e^{-\frac{\gamma}{2}}}e^{\frac{\gamma}{2}(t-t_0)},\, \forall t\leq t_0,
\end{equation}
while according to \eqref{formula} we have
\begin{equation}\label{cvv2}
\| u_x(t,\cdot)-(e^-)'\|_{L^2(\R:\R^m)}^2\leq 2\mathcal W(U(t))+C\| U(t)-e^-\|_{L^2(\R:\R^m)}^2.
\end{equation}
Gathering the previous results, we deduce that $\| U(t)-e^-\|_{H^1(\R:\R^m)}$ converges exponentially to $0$.
\end{proof}

\section{Proof of Theorem \ref{connh3}}
To prove the existence of the minimizer $\tilde U$, just replace in the proof of Theorem \ref{connh2}, $\mathcal H$, $d_{\mathrm{min}}$, $\mathcal J$ and $\mathcal A$, by $\mathcal{\tilde H}$, $\tilde d_{\mathrm{min}}$, $\mathcal{\tilde J}$ and $\mathcal{\tilde A}$.
Next, given a function $\Phi \in C^1_0(\R;H^1(\R;\R^m))$ such that $\supp\Phi \subset [\alpha,\beta]\subset\R$, it is clear that for every $\lambda \in \R$, we have
\begin{equation}\label{var1til}
\mathcal{\tilde J}_\R(\tilde U)\leq\mathcal{\tilde J}_\R(\tilde U+\lambda\Phi),
\end{equation}
and
\begin{equation}\label{var33til}
\frac{\dd}{\dd\lambda}\Big|_{\lambda=0}\int_\R\frac{1}{2}\|\tilde U'(t)+\lambda \Phi'(t)\|_{H^1(\R;\R^m)}^2\dd t=\int_\R\langle\tilde U'(t),\Phi'(t)\rangle_{H^1(\R;\R^m)}\dd t.
\end{equation}
On the other hand, proceeding as in the proof of Theorem \ref{connh2} we obtain
\begin{equation}\label{var22til}
\frac{\dd}{\dd\lambda}\Big|_{\lambda=0}\int_\R\mathcal W(\tilde U(t)+\lambda\Phi(t))\dd t=\int_{\R}\psi(t)\dd t,
\end{equation}
with $\psi(t)=\int_\R\big[\frac{\dd [\tilde U(t)]}{\dd x}\cdot \frac{\dd [\Phi(t)]}{\dd x}+\nabla W([\tilde U(t)](x))\cdot[\Phi(t)](x)\big]\dd x=D\mathcal W(\tilde U(t)) \Phi(t)$ (cf. Lemma \ref{lem1w} (iii)). Indeed, in view of \eqref{cvboundd}, the functions
$\psi_\lambda(t):=\frac{1}{\lambda}[\mathcal W(\tilde U(t)+\lambda \Phi(t))-\mathcal W(\tilde U(t))]$ converge as $\lambda\to 0$ to $\psi(t)$, and are uniformly bounded when $|\lambda|\leq 1$, by the integrable function
\begin{multline*}
\Psi(t)=(\|\tilde U(t)\|_{\mathcal{\tilde H}}+\|\ee'_0\|_{L^2(\R;\R^m)}+\|\Phi(t)\|_{\mathcal{\tilde H}})\|\Phi(t)\|_{\mathcal{\tilde H}}+\kappa_1( \|\tilde U(t)\|_{\mathcal{\tilde H}} + \|\Phi(t)\|_{\mathcal{\tilde H}})\|\Phi(t)\|_{\mathcal{\tilde H}}+2\kappa \kappa_2\chi_{[\alpha,\beta]}(t),
\end{multline*}
where $\kappa=\sup_{t\in[\alpha,\beta]}(\|\tilde U(t)\|_{L^\infty(\R;\R^m)} + \|\Phi(t)\|_{L^\infty(\R;\R^m)})$, $\kappa_1=\sup_{|u|\leq \kappa, |\nu|=1} |D^2W(u)(\nu,\nu)|$, $\kappa_2=\sup_{|u|\leq\kappa}|\nabla W(u)|$, and $\chi$ is the characteristic function. Gathering the previous results we conclude that the minimizer $\tilde U$ satisfies the Euler-Lagrange equation
\begin{equation}\label{euler1til}
\int_{\R}(\langle \tilde U'(t),\Phi'(t)\rangle_{H^1(\R;\R^m)}+D\mathcal W(\tilde U(t)) \Phi(t)   )\dd t=0,
\end{equation}
and thus $\tilde U\in C^2(\R;\mathcal{\tilde H})$ is a classical solution of system $\tilde U''=\nabla \mathcal W(\tilde U)$. Next, we notice that the space $L^2((\alpha,\beta);H^1(\R;\R^m))$ is imbedded in $L^2((\alpha,\beta);L^2(\R;\R^m))$ which is isomorphic to $L^2((\alpha,\beta)\times\R;\R^m)$. Similarly, the space $H^1((\alpha,\beta);H^1(\R;\R^m))$ is imbedded in $H^1((\alpha,\beta);L^2(\R;\R^m))$, thus Lemma \ref{sobolev} also applies to $\tilde U$. That is, setting $\tilde u(t,x):=[\tilde U(t)](x)$, $t\mapsto \tilde U(t)\in \mathcal{\tilde H}$, we have $\tilde u \in H^1_{\mathrm{loc}}(\R^2;\R^m)$, $\tilde u_t\in L^2(\R^2;\R^m)$, and $\tilde u_x\in L^2((\alpha,\beta)\times \R;\R^m)$. Furthermore, we can see that $\tilde u_{tx}\in L^2(\R^2;\R^m)$ by using difference quotients as in the proof of Lemma \ref{sobolev}. In view of the previous results, \eqref{euler1til} and \eqref{var1til} read respectively \eqref{wpde} and \eqref{wmin}, when $\phi(t,x):=[\Phi(t)](x)$ is a $C^2_0(\R^2;\R^m)$ function. To prove \eqref{lay2btil}, we notice that $\tilde u$ is uniformly continuous on the strips $[\alpha,\beta]\times \R$, since $[\alpha,\beta]\ni t\mapsto \tilde U(t)\in \mathcal{\tilde H}$ is Lipschitz continuous, and $|\tilde u(t,x)-\tilde u(t,y)|\leq \lambda|x-y|^{\frac{1}{2}}$ holds for $t\in [\alpha,\beta]$, $x,y\in\R$, and $\lambda=\sup_{[\alpha,\beta]}\|\tilde U(t)\|_{\mathcal{\tilde H}}$. Then, we establish \eqref{lay2btil}, \eqref{lay1btil} and the equipartition property \eqref{equi22} as in the proof of Theorem \ref{connh2}. Finally, when $\mathcal W$ satisfies the nondegeneracy condition \eqref{nondegWtil}, the arguments in the proof of Theorem \ref{connh2} still apply to show \eqref{lay1newtil}, since we have $\sup\{\|e\|_{L^\infty(\R;\R^m)} : e\in F\}<\infty$ as well as $\sup_{t\in\R}\|\tilde U(t)\|_{L^\infty(\R;\R^m)}<\infty$. On the other hand, it is clear in view of \eqref{lay1newtil} that the uniform convergence in \eqref{lay2btil} holds for $t\in\R$.

\section*{Acknowledgments}

The author was partially supported by the National Science Centre, Poland (Grant No. 2017/26/E/ST1/00817)

\bibliographystyle{plain}

\begin{thebibliography}{99}

\bibitem{abg}
Alama, S., Bronsard, L., Gui, C.:
Stationary layered solutions in $\R^2$ for an Allen-Cahn system with multiple well potential.
Calc.\ Var. \textbf{5} No.~4, 359--390 (1997)



\bibitem{alessio}
Alessio, F.: Stationary layered solutions for a system of Allen-Cahn type equations.
Indiana Univ. Math. g. \textbf{62}, 1535--1564 (2013)

\bibitem{alessio2}
Alessio, F., Montecchiari, P.: Brake orbit solutions for semilinear elliptic systems with asymmetric double well potential. J. Fixed Point Theory Appl. \textbf{19} (1), 691--717 (2017)

\bibitem{alessio3}
Alessio, F. G., Montecchiari, P., Zuniga, A.: Prescribed energy connecting orbits for gradient systems.
Discrete Cont. Dyn. Syst. Ser. A \textbf{39} (8), 4895--4928 (2019)




\bibitem{alikakos1}
Alikakos, N. D., Fusco, G.:
On the connection problem for potentials with several global minima.
Indiana Univ. Math. J. \textbf{57}, 1871--1906 (2008)

\bibitem{alikakos2}
Alikakos, N. D., Fusco, G.:
Density estimates for vector minimizers and applications.
Discrete and continuous dynamical systems \textbf{35} No.~12, 5631--5663 (2015), Special issue edited by E.Valdinoci

\bibitem{antonop}
Antonopoulos, P., Smyrnelis, P.: On minimizers of the Hamiltonian system $u''=\nabla W(u)$, and on the existence of heteroclinic, homoclinic and periodic orbits. Indiana Univ.\ Math.\ J. \textbf{65} No.~5, 1503--1524 (2016)

\bibitem{brezis2}
Brezis, H.: Op\'erateurs maximaux monotones et semi-groupes de contractions dans les espaces de Hilbert. \textbf{50} in Notas de Matem\'atica. North-Holland Publishing Company (1973)

\bibitem{brezis}
Brezis, H.: Functional Analysis, Sobolev Spaces and Partial Differential Equations.
Universitext, Springer-Verlag, New York (2011)

\bibitem{caf}
Caffarelli, L., C\'{o}rdoba, A.: Uniform convergence of a singular perturbation problem.
Comm. Pure Appl. Math. \textbf{48}, 1--12 (1995)

\bibitem{caz}
Cazenave, T., Haraux, A.: An Introduction to Semilinaer Evolution Equations. Oxford lecture series in mathematics and its applications. Clarendon Press (1998)

\bibitem{evans}
Evans, L. C.:
Partial differential equations.
Graduate Studies in Mathematics \textbf{19}, American Mathematical Society, second edition (2010)

\bibitem{fusco}
Fusco, G.:
Layered solutions to the vector Allen-Cahn equation in  $\R^2$. Minimizers and heteroclinic connections.
Comm. Pure Appl. Anal. \textbf{16} No.~5, 1807--1841 (2017)


\bibitem{fusco2}
Fusco, G., Gronchi, G. F., Novaga, M.: On the existence of heteroclinic connections. Sao Paulo
J. Math. Sci. \textbf{12}, 1--14 (2017)

\bibitem{pap}
Gasinski, L., Papageorgiou, N. S.: Nonlinear Analysis. Series in Mathematical Analysis and Applications. CRC Press (2006)

\bibitem{gilbarg}
Gilbarg, D., Trudinger, N. S.: Elliptic partial differential equations of second order. Grundlehren der mathematischen Wissenschaften {\bf 224}, Springer-Verlag, Berlin, revised second edition, (1998)

\bibitem{kreuter}
Kreuter, M.: Sobolev spaces of vector-valued functions. Master thesis, Ulm University, Faculty of Mathematics and Economics (2015)


\bibitem{monteil}
Monteil, A., Santambrogio, F.: Metric methods for heteroclinic connections in infinite dimensional spaces. To appear.

\bibitem{savin}
Savin, O.: Minimal Surfaces and Minimizers of the Ginzburg Landau energy. Cont. Math. Mech. Analysis AMS \textbf{526}, 43--58 (2010)

\bibitem{scha}
Schatzman, M.:
Asymmetric heteroclinic double layers.
Control Optim. Calc. Var. \textbf{8} (A tribute to J. L. Lions), 965--1005 (electronic) (2002)

\bibitem{ps}
Smyrnelis, P.: Minimal heteroclinics for a class of fourth order O.D.E. systems.
Nonlinear Analysis, Theory, Methods and Applications \textbf{173}, 154--163 (2018)


\bibitem{stern}
Sternberg, P.: The effect of a singular perturbation on nonconvex variational problems.
Arch. Rational Mech. Anal. \textbf{101} No. 3, 209--260 (1988)

\bibitem{stern2}
Sternberg, P., Zuniga, A.: On the heteroclinic connection problem for multi-well gradient systems.
Journal of Differential Equations \textbf{261} No. 7, 3987--4007 (2016)

\end{thebibliography}

\end{document}